\documentclass{amsproc}
\usepackage{eurosym}
\usepackage{amssymb}
\usepackage{amsfonts}
\usepackage{cmtt}
\usepackage{color}
\usepackage{graphicx}					

\usepackage{calrsfs}
\DeclareMathAlphabet{\pazocal}{OMS}{zplm}{m}{n}

\DeclareMathOperator{\modul}{mod}
\DeclareMathOperator{\vol}{vol}			
\theoremstyle{plain}
\newtheorem{theorem}{Theorem}

\newtheorem{conjecture}[theorem]{Conjecture}
\newtheorem{corollary}[theorem]{Corollary}
\newtheorem{definition}[theorem]{Definition}

\newtheorem{lemma}[theorem]{Lemma}

\newtheorem*{remark*}{Remark}

\renewcommand{\Pr}{\mathbb{P}}
\newcommand{\E}{\mathbb{E}}
\newcommand{\var}[1]{\mathrm{Var}[#1]}

\theoremstyle{remark}

\newcommand{\superscript}[1]{\ensuremath{^{\textrm{#1}}}}

\def\wu{\superscript{*}}
\def\wg{\superscript{$\star$}}

\begin{document}

\title[Modularity of preferential attachment graphs]{Modularity of preferential attachment graphs} 

\author[K.~Rybarczyk]{Katarzyna Rybarczyk\wu\footnote{\wu kryba@amu.edu.pl, Adam Mickiewicz University, Pozna\' n, Poland}} 

\author[M.~Sulkowska]{Ma{\l}gorzata Sulkowska\wg\footnote{\wg malgorzata.sulkowska@pwr.edu.pl,
Wroc{\l}aw University of Science and Technology, Department of Fundamentals of Computer Science, Poland}}



\begin{abstract} 
We study a preferential attachment model $G_n^h$. The graph $G_n^h$ is generated from a finite initial graph by adding new vertices one at a time. Each new vertex connects to $h\ge 1$ already existing vertices, and these are chosen with probability proportional to their current degrees. We are particularly interested in the community structure of $G_n^h$, which is expressed in terms of the so--called modularity. We prove that the modularity of $G_n^h$ is, with high probability, upper bounded by a function that tends to $0$ as $h$ tends to infinity. This resolves a conjecture of Prokhorenkova, Pra{\l}at, and Raigorodskii from 2016.

As a byproduct, we obtain novel concentration results (which are interesting in their own right) for the volume and edge density parameters of vertex subsets of $G_n^h$. The key ingredient here is the definition of a function $\mu$, which serves as a natural measure for vertex subsets, and is proportional to the average size of their volumes. This extends previous results on the topic by Frieze,  P{\'e}rez-Gim{\'e}nez, Pra{\l}at, and Reiniger from 2019.

\end{abstract} 

\maketitle 

\section{Introduction}

Real-world networks, ranging from social and information networks to biological and technological infrastructures, often exhibit a rich community structure. Detecting and analyzing such communities has far-reaching applications: identifying groups of common interest in social media, classifying spam and misinformation, retrieving related content, uncovering proteins with similar biological functions, optimizing large-scale infrastructures, improving network visualization, etc.~\cite{KaPrTh21,NeBook10}.

To model these networks mathematically, preferential attachment graphs have become one of the main paradigms. Their early forms appeared as random recursive trees~\cite{Szymanski1993,Szymanski1987}. However, the in-depth study of preferential attachment models was initiated in 1999 by the work of Barabási and Albert \cite{BA_basic}, who indicated the applications of such graphs in network modeling. 
The preferential attachment model was subsequently formally defined and analyzed by Bollob\'as,  Riordan, Spencer, and Tusn\'ady \cite{Bollobas2001}, and Bollob\'as and Riordan in \cite{Bol_Rio_chapter,Bol_Rio_diameter}. It relies on two mechanisms: growth (the graph is growing over time, gaining a new vertex and a bunch of $h \geq 1$ edges at each time step) and preferential attachment (an arriving vertex is more likely
to attach to other vertices with high degree rather than with low degree), for a precise definition check Section \ref{sec:notation}. Its degree distribution as well as diameter often fit in with the ones spotted in reality \cite{Hofstad_2017,NeBook10}. Nevertheless, an experimental study shows that, unlike  real networks, it lacks apparent community structure. 


Quantifying community structure itself is a subtle task. Among the many measures proposed, modularity, introduced by Newman and Girvan in 2004~\cite{NeBook10,NeGi04}, has emerged as a central metric. The vertices of a graph with high modularity may be partitioned into subsets in which there are much more internal edges than we would expect by chance (see Definition~\ref{def:modularity}). Nowadays, modularity is widely used not only as a quality function judging the performance of community detection algorithms~\cite{KaPrTh21}, but also as a central ingredient of such algorithms, like in the Louvain algorithm \cite{BlGuLaLe08}, the Leiden algorithm \cite{TrWaEc19} or the Tel-Aviv algorithm \cite{GiSh_23}. Early theoretical results on modularity were given for trees \cite{Bagrow_trees_12} and regular graphs \cite{McSk_reg_lattice_13}. For a summary of results for various families of graphs check the appendix of \cite{DiSk20} by McDiarmid and Skerman from $2020$.
More recent discoveries include \cite{ChFoSk21} by Chellig, Fountoulakis, and Skerman (for random graphs on the hyperbolic plane), \cite{LaSu23} by Laso{\'n} and Sulkowska (for minor-free graphs), \cite{LiMi22} by Lichev and Mitsche (for $3$-regular graphs and graphs with a given degree sequence), or \cite{Rybarczyk2025randomintersectiongraphs} by Rybarczyk (for random intersection graphs).

Despite being so widely used in practice, modularity still suffers from a narrow theoretical study in the families of random graphs devoted to modeling real-life networks. The first results for the well-known and most studied random graph, the binomial $G(n,p)$, were given by McDiarmid and Skerman just in 2020 \cite{DiSk20}. It is commonly known that $G(n,p)$ is a poor fit to real networks~\cite{KaPrTh21}. Preferential attachment models perform here much better. 
Prokhorenkova, Pra{\l}at, and Raigorodskii opened the preliminary study on modularity of a standard preferential attachment graph in \cite{prokhorenkova2017modularity_internet_Math}. They obtained non-trivial upper and lower bounds, however, the gap to close remained big. They conjectured that the modularity of such a graph with high probability tends to $0$ with $h$ (the number of edges added per step) tending to infinity (see Conjecture \ref{conj:mod->0}). In this paper we prove their conjecture, confirming the supposition that a standard preferential attachment model might have too small modularity to mirror well the behavior of real networks. 

As a result, we derive new and interesting concentration results for the volume and edge density parameters of a given subset of vertices in the preferential attachment graph. To this end, we introduce a new function $\mu$, which serves as a natural measure for vertex subsets, and is proportional to the average size of their volumes. These findings are noteworthy on their own and could have potential applications to other problems related to the model in the future. They extend previous results from \cite{FrPePrRe20} by Frieze,  P{\'e}rez-Gim{\'e}nez, Pra{\l}at, and Reiniger (see lemmas~3 and~4 therein), which were utilized in the context of Hamilton cycles in the preferential attachment model.

In the following section we give the formal definition of the preferential attachment model and state the main result. Section \ref{sec:VolDensity} is devoted to presenting the results regarding the volume and the edge density parameters of subsets of vertices in $G_n^h$. Section~\ref{sec:auxliary} is technical, it contains several facts and auxiliary lemmas used in the latter parts of the paper. In Section~\ref{sec:concentrations} we derive concentration results stated in Section \ref{sec:VolDensity}. These results are used in Section~\ref{sec:modularity} to prove the main theorem about vanishing modularity in the standard preferential attachment graph. Section \ref{sec:conclusion} contains concluding remarks. 


\section{Model and main result} \label{sec:notation}

Let $\mathbb{N}$ denote the set of natural numbers, $\mathbb{N} = \{1,2,3,\ldots\}$. For $n \in \mathbb{N}$ let $[n]=\{1,2,\ldots,n\}$. For functions $f(n)$ and $g(n)$ we write $f(n) \sim g(n)$ if $\lim_{n \rightarrow \infty} f(n)/g(n) = 1$. We say that an event $\mathcal{E}$ occurs with high probability (whp) if the probability $\Pr[\mathcal{E}]$ depends on a certain number $n$ and tends to $1$ as $n$ tends to infinity.

All of the graphs considered in this paper are finite, undirected, and loops and multiple edges are allowed. Thus a graph is a pair $G=(V,E)$, where $V$ is a finite set of vertices and $E$ is a finite multiset of elements from $V^{(1)} \cup V^{(2)}$ with $V^{(k)}$ being a set of all $k$-element subsets of $V$. 
Let $e(G) = |E|$ and for $S,U \subseteq V$ set $e_G(S) = |\{e \in E \cap (S^{(1)} \cup S^{(2)})|$ and $e_G(S,U) = |\{e \in E: e \cap S \neq \emptyset \wedge e \cap U \neq \emptyset\}|$. The degree of a vertex $v \in V$ in $G$, denoted by $\deg_G(v)$, is the number of edges to which $v$ belongs but loops are counted twice, i.e., $\deg_G(v) = 2 |\{e \in E: v \in e \wedge e \in V^{(1)}\}| + |\{e \in E: v \in e \wedge e \in  V^{(2)}\}|$. We define the volume of $S \subseteq V$ in $G$ by $\vol_G(S) = \sum_{v \in S} \deg_G(v)$. By the volume of a graph, $\vol(G)$, we understand $\vol_G(V)$.
Whenever the context is clear we write $e(S)$ instead of $e_G(S)$, $e(S,U)$ instead of $e_G(S,U)$, $\deg(v)$ instead of $\deg_G(v)$ and $\vol(S)$ instead of $\vol_G(S)$.

We focus on a particular random graph model, called here simply the \textit{preferential attachment graph} (consult \cite{BA_basic,Bol_Rio_chapter,Hofstad_2017}). Given $h,n \in \mathbb{N}$, we construct a preferential attachment graph $G_n^h$ in two phases. In the first phase we sample a particular random tree $T_{hn}$, whose vertices are called mini-vertices. (We call $T_{hn}$ a tree, however it might be disconnected, and loops, i.e. single-vertex edges, are allowed in $T_{hn}$.) Next, the appropriate mini-vertices of $T_{hn}$ are grouped to form vertices of $G_n^h$. Let us describe this procedure in detail.\\
\textbf{Phase 1.} We start the whole process with $T_1$ which is a graph consisting of a single mini-vertex $1$ with a single loop (thus the degree of vertex $1$ is $2$). For $t \geq 1$, the graph $T_{t+1}$ is built upon $T_{t}$ by adding a mini-vertex $(t+1)$ and joining it by an edge with a mini-vertex $i$ according to the following probability distribution:
\[
\Pr(i=s) = \begin{cases}  \frac{\deg_{T_t}(s)}{2t+1} & \textnormal{for} \quad 1 \leq s \leq t  \\
	\frac{1}{2t+1} & \textnormal{for} \quad s=t+1.
\end{cases}
\]
Note that we allow a newly arrived vertex to connect to itself. We continue the process until we get the random tree $T_{hn}$.\\
\textbf{Phase 2.} A random multigraph $G_n^h$ is obtained from $T_{hn}$ by merging each set of mini-vertices $\{h(i-1)+1, h(i-1)+2, \ldots, h(i-1)+h\}$ into a single vertex $i$ for $i \in \{1,2, \ldots, n\}$, keeping loops and multiple edges.

Note that if $G_n^h=(V,E)$ then $V=[n]$, $|V|=n$ and $|E|=hn$. Since we will refer very often to the number of edges of $G_n^h$, it will be also denoted by $M$, i.e., $M:=hn$. Given $G_n^h$, by $G_t^h$, where $t\in[n]$, we understand the subgraph of $G_n^h$ induced by the set of vertices $[t]$.

Our main goal is to upper bound the graph parameter called \textit{modularity} for~$G_n^h$. 
Its formal definition is given just below.
\begin{definition}[Modularity, \cite{NeGi04}] \label{def:modularity}
	Let $G$ be a graph with at least one edge. For a partition $\mathcal{A}$ of $V$ define a modularity score of $G$ as
	\[
	\modul_{\mathcal{A}}(G) = \sum_{S\in\mathcal{A}}\left(\frac{e(S)}{e(G)}-\left(\frac{\vol(S)}{\vol(G)}\right)^2\right).
	\]
	Modularity of $G$ is given by
	\[
	\modul(G) = \max_{\mathcal{A}}\modul_{\mathcal{A}}(G),
	\]
	where maximum runs over all the partitions of the set $V$.
\end{definition}

Conventionally, a graph with no edges has the modularity equal to $0$. A single summand of the modularity score is the difference between the fraction of edges within $S$ and the expected fraction of edges within $S$ in a certain random multigraph on $V$ with the expected degree sequence given by $G$ (see, e.g.,~\cite{Pralat_hyper_modularity}). It is easy to check that $\modul(G) \in [0,1)$. 

Non-trivial lower and upper bounds for the modularity of $G_n^h$ obtained by Prokhorenkova, Pra{\l}at and Raigorodskii in \cite{prokhorenkova2017modularity_internet_Math} are the following.

\begin{theorem} [\cite{prokhorenkova2017modularity_internet_Math}, Theorem 4.2, Section 4.2] \label{thm:Prokh_Gnh_modularity}
	Let $G_n^h = (V,E)$ be a preferential attachment graph. Then whp by $n \to \infty$
	\[
	\modul(G_n^h) = \Omega_h(1/\sqrt{h})
	\]
	and whp by $n \to \infty$
	\[
	\modul(G_n^h) \leq 1-\min\{\delta(G_n^h)/(2h),1/16\},
	\]
	where $\delta(G_n^h) = \min\limits_{\substack {S \subseteq V,  1\leq|S|\leq |V|/2}} \frac{e(S,V \setminus S)}{|S|}$ is the edge expansion of $G_n^h$.
\end{theorem}
Applying the results for the edge expansion of $G_n^h$ by  Mihail, Papadimitriou and Saberi from \cite{MiPaSa06_journalv} to the upper bound one obtains that whp $\modul(G_n^h) \leq 1 - O(1/h)$. Indeed, the gap between the upper and the lower bound remained big. The authors stated the following two conjectures suggesting that the upper bound could be improved. 
\begin{conjecture} [\cite{prokhorenkova2017modularity_internet_Math}] \label{conj:mod->0}
	Let $G_n^h$ be a preferential attachment graph. Then whp by $n \to \infty$ \[ \modul(G_n^h)~\xrightarrow{h \rightarrow \infty}~0.\]
\end{conjecture}
\begin{conjecture} [\cite{prokhorenkova2017modularity_internet_Math}] \label{conj:mod_theta}
	Let $G_n^h$ be a preferential attachment graph. Then whp by $n \to \infty$ \[\modul(G_n^h) = \Theta_h(1/\sqrt{h}).\]
\end{conjecture}

In this paper we present a much better upper bound for the modularity of $G_n^h$ than the one from Theorem \ref{thm:Prokh_Gnh_modularity} when $h$ is large, resolving, in the positive, Conjecture \ref{conj:mod->0}. Conjecture \ref{conj:mod_theta} still remains open. The main result of the paper may be presented as follows.
\begin{theorem} \label{thm:main}
	Let $G_n^h$ be a preferential attachment graph. Then for every $\varepsilon>0$, whp by $n \to \infty$
	\[
	\modul(G_n^h) \leq \frac{(1+\varepsilon)f(h)}{\sqrt{h}},
	\]
	where 
	\[
	f(h) = 6 g_{\mathcal{V}}(h) + 4 \sqrt{2 \ln{2}} - g_{\mathcal{V}}(h)^2/\sqrt{h}
	\]
	with
	\[
	g_{\mathcal{V}}(h) = \frac{1}{6}\sqrt{2 \ln{2} \, (9 \ln{h}+8\ln{2})} + (2/3) \ln{2} + 2.
	\]
\end{theorem}

\begin{remark*}
	Note that $f(h) \sim 3 \sqrt{2\ln{2}} \sqrt{\ln{h}}$ as $h \rightarrow \infty$ thus
	$f(h)/\sqrt{h} \rightarrow 0$ as $h \rightarrow \infty$. The value of $f(h)/\sqrt{h}$ drops below $1$ for $h \geq 810$. 
\end{remark*}

\begin{corollary} \label{cor:main}
	Let $G_n^h$ be a preferential attachment graph. Then whp by $n \to \infty$
	\[
	\modul(G_n^h) \leq \frac{3.54 \sqrt{\ln{h}+0.62} + 19.49 }{\sqrt{h}}.
	\]
\end{corollary}
\begin{remark*}
	The value $\frac{3.54 \sqrt{\ln{h}+0.62} + 19.49 }{\sqrt{h}}$ drops below $1$ for $h \geq 847$.
\end{remark*}

\begin{remark*}
	Some new results on the fact that $\modul(G_n^h)$ is whp separated from 1 by a constant even for small values of $h$ can be found in \cite{McDRSS_2024}.
\end{remark*}


\section{Volume and edge density} \label{sec:VolDensity} 

When talking about $G_n^h = (V,E)$ we will very often refer to its corresponding random tree $T_{hn}=(\tilde{V},\tilde{E})$. Recall that $V=[n]$, $\tilde{V} = [hn]$ and $|\tilde{E}|=|E|=hn=:M$. For $S \subseteq V$ the corresponding set of its mini-vertices in $\tilde{V}$ will be denoted by $\tilde{S}$, thus $|\tilde{S}| = h |S|$. For $i \in [n]$ and $S \subseteq V$ let $S_i = S \cap [i]$, in particular $S_n = S$. Analogously, for $i \in [M]$ and $\tilde{S} \subseteq \tilde{V}$ set $\tilde{S}_i = \tilde{S} \cap [i]$, in particular $\tilde{S}_M = \tilde{S}$. Note that for $S \subseteq V$ we have $\vol_{G_n^h}(S) = \vol_{T_{hn}}(\tilde{S})$ and $e_{G_n^h}(S) = e_{T_{hn}}(\tilde{S})$.

When working with modularity we need to have a control over $e_{G_n^h}(S)$ and $\vol_{G_n^h}(S)$, where $S \subseteq V$. Those values depend a lot on the arrival times of vertices from $S$. To capture this phenomenon we define a special measure $\mu:2^{\tilde{V}} \rightarrow [0,\infty)$, where $2^{\tilde{V}}$ stands for the set of all subsets of $\tilde{V}$.

\begin{definition}[Measure $\mu$] \label{def:mu_delta_c}
	Let $G_n^h = (V,E)$ be a preferential attachment graph and $T_{hn} = (\tilde{V},\tilde{E})$ its corresponding random tree. Let $S \subseteq V$ thus $\tilde{S} \subseteq \tilde{V}$ is the set of its corresponding mini-vertices. Associate $\tilde{S}$ with the set of indicator functions
	\[
	{\delta}_i^{\tilde{S}} = \begin{cases} 1 &\quad \textnormal{if} \quad i \in \tilde{S}\\
		0 &\quad \textnormal{if} \quad i \notin \tilde{S},
	\end{cases}
	\]
	where $i \in [M]$ (whenever the context is clear we write ${\delta}_i$ instead of ${\delta}_i^{\tilde{S}}$). Define a function $\mu:2^{\tilde{V}} \rightarrow [0,\infty)$ as follows:
	\[
	\mu(\tilde{S}) = \frac{\sqrt{\pi}}{2} \cdot \sum_{j=1}^{M} {\delta}_j^{\tilde{S}} c_{j-1}
	\]
	with $c_j = \prod_{i=1}^{j}\frac{2i-1}{2i}$ for $j \geq 1$ and $c_0=1$.
\end{definition}
\begin{remark*}
	Let $G_n^h = (V,E)$ be a preferential attachment graph, $S \subseteq V$ and $t \in [M]$. Note that
	\[
	\mu(\tilde{S}_t) = \frac{\sqrt{\pi}}{2} \cdot \sum_{j=1}^{t} {\delta}_j^{\tilde{S}} c_{j-1}.
	\]
\end{remark*}

We use the measure $\mu$ to express the following novel concentration results for $\vol_{G_n^h}(S)$, $e_{G_n^h}(S)$, and $e_{G_n^h}(S,V \setminus S)$, where $S$ is an arbitrary subset of $V$.

\begin{theorem} \label{thm:vol_concentration}
	Let $G_n^h = (V,E)$ be a preferential attachment graph and $T_{hn} = (\tilde{V},\tilde{E})$ its corresponding random tree. Then for every $\varepsilon >0$ whp by $n \to \infty$
	\[ 
	\forall S \subseteq V \,\,
	\left|\vol(S) - 2 \sqrt{M} \, \mu(\tilde{S})\right| \leq  (1+\varepsilon)g_{\mathcal{V}}(h) \frac{M}{\sqrt{h}}, 
	\]
	where $g_{\mathcal{V}}(h) = \frac{1}{6}\sqrt{2 \ln{2} \, (9 \ln{h}+8\ln{2})} + (2/3) \ln{2} + 2$.
\end{theorem}

\begin{theorem} \label{thm:edges_concentration}
	Let $G_n^h$ be a preferential attachment graph and $T_{hn} = (\tilde{V},\tilde{E})$ its corresponding random tree. Then for every $\varepsilon >0$ whp by $n \to \infty$ 
	\[
	\forall S \subseteq V \,\, \left| e(S) - \mu(\tilde{S})^2 \right| \leq (1+\varepsilon)g_{\mathcal{E}}(h) \frac{M}{\sqrt{h}},
	\]
	where 
	\[
	g_{\mathcal{E}}(h) = \frac{g_{\mathcal{V}}(h)}{2} + \sqrt{2 \ln{2}}
	\]
	with
	\[
	g_{\mathcal{V}}(h) = \frac{1}{6}\sqrt{2 \ln{2} \, (9 \ln{h}+8\ln{2})} + (2/3) \ln{2} + 2.
	\]
\end{theorem}

\begin{theorem} \label{thm:edges_between_concentration}
	Let $G_n^h$ be a preferential attachment graph and $T_{hn} = (\tilde{V},\tilde{E})$ its corresponding random tree. Then for every $\varepsilon >0$ whp by $n \to \infty$ 
	\[
	\forall S \subseteq V \, \left| e(S, V \setminus S) - 2\mu(\tilde{S})(\sqrt{M}-\mu(\tilde{S})) \right| \leq (1+\varepsilon) \left(\frac{3}{2}g_{\mathcal{V}}(h) + \sqrt{2\ln{2}}\right) \frac{M}{\sqrt{h}},
	\]
	where 
	\[
	g_{\mathcal{V}}(h) = \frac{1}{6}\sqrt{2 \ln{2} \, (9 \ln{h}+8\ln{2})} + (2/3) \ln{2} + 2.
	\]
\end{theorem}

To grasp the intuition hidden behind the above concentration results, it is helpful to know that there is a relation between the structure of the graph $T_{hn}$  and the structure of a random graph $\hat{G}$ on the vertex set $[M]$ in which every edge $\{i,j\}$ (for $i,j \in [M]$) is present
with probability $1/(2\sqrt{ij})$, independently of the other possible edges (in particular, a loop at vertex $i$ is present with probability $1/(2i)$, consult Section 4 in \cite{Bol_Rio_diameter} by Bollob{\'a}s and Riordan). 
We will see that, for any set $\tilde{S}$, the values $\mu(\tilde{S})$ and $\mu(\tilde{S})^2$ are closely related to the expected value of $\vol_{\hat{G}} (\tilde{S})$ and $e_{\hat{G}}(\tilde{S})$, respectively.

Let $\tilde{S} \subseteq [M]$. The number of inner edges of $\tilde{S}$ in $\hat{G}$, $e_{\hat{G}}(\tilde{S})$, satisfies
\begin{align*}
	\E[e_{\hat{G}}(\tilde{S})] & = \frac{1}{2} \sum_{i \in\tilde{S}} \sum_{\substack{j \in S \\ j \neq i}} \frac{1}{2\sqrt{ij}} + \sum_{i \in \tilde{S}}\frac{1}{2i} =  \sum_{i \in \tilde{S}} \frac{1}{2\sqrt{i}} \sum_{\substack{j \in \tilde{S} \\ j \neq i}} \frac{1}{2\sqrt{j}} + \sum_{i \in \tilde{S}}\frac{1}{2i} \\
	& =  \left(\sum_{i\in \tilde{S}} \frac{1}{2\sqrt{i}}\right)^2 - \sum_{i \in \tilde{S}} \frac{1}{4i} = \left(\sum_{i\in \tilde{S}} \frac{1}{2\sqrt{i}}\right)^2 + O(\ln{M}).
\end{align*}
Analogously one shows that the number of edges between $\tilde{S}$ and $[M]\setminus \tilde{S}$, $e_{\hat{G}}(\tilde{S},[M]\setminus \tilde{S})$, fulfills
$$\E[e_{\hat{G}}(\tilde{S},[M]\setminus \tilde{S})] = 2 \left(\sum_{i\in \tilde{S}} \frac{1}{2\sqrt{i}}\right) \left(\sum_{i\in [M]\setminus \tilde{S}} \frac{1}{2\sqrt{i}}\right).$$
Since $\vol_{\hat{G}} (\tilde{S}) = 2 e_{\hat{G}}(\tilde{S}) + e_{\hat{G}}(\tilde{S},\tilde{V}\setminus{\tilde{S}})$ one also gets
$$\E[\vol_{\hat{G}} (\tilde{S}) ] = 2  \left(\sum_{i\in [M]} \frac{1}{2\sqrt{i}}\right) \left(\sum_{i\in \tilde{S}} \frac{1}{2\sqrt{i}}\right) +O(\ln{M}).$$

We will see later (consult Lemma~\ref{lemma:ct}) that the value of $\sqrt{\pi}c_j$ is asymptotically close to $1/\sqrt{j}$ thus the measure $\mu$ is constructed in such a way that $\mu(\tilde{S})$ mimics the behavior of $\sum_{i \in \tilde{S}} \frac{1}{2\sqrt{i}}$ in $\hat{G}$.  In particular $\mu(\{i\}) \sim \frac{1}{2 \sqrt{i}}$ for $i \xrightarrow[n \rightarrow \infty]{} \infty$ and $\mu([M]) \sim \sqrt{M}$. Therefore we may expect that in $T_{hn}$ we will get $e(\tilde{S}) \approx \mu(\tilde{S})^2$, $e(\tilde{S}, V\setminus \tilde{S}) \approx 2\mu(\tilde{S})(\sqrt{M}-\mu(\tilde{S}))$ and $\vol(\tilde{S}) \approx 2 \sqrt{M} \mu(\tilde{S})$.


\section{Auxiliary lemmas} \label{sec:auxliary}

The current section gathers all technical lemmas needed in the latter parts of the paper.

The concentration results presented in Section \ref{sec:VolDensity} and proved in Section \ref{sec:concentrations} are based on two variants of the Azuma-Hoeffding martingale inequality. The first one is standard. We state it as it appears in \cite{JLR_random_graphs} by Janson, {\L}uczak and Ruci{\'n}ski.

\begin{lemma} [Azuma-Hoeffding inequality, \cite{Az67,Ho63}] \label{lemma:Azuma}
	If $X_0, X_1, \ldots, X_n$ is a martingale and there exist $b_1, \ldots, b_{n}$ such that $|X_{j} - X_{j-1}| \leq b_j$ for each $j \in [n]$, then, for every $x>0$,
	\[
	\Pr[X_n \geq X_0
	+x] \leq \exp\left\{-\frac{x^2}{2 \sum_{j=1}^{n}b_j^2} \right\}.
	\]
\end{lemma}

The second one is Freedman's inequality. We state it below in the form very similar to the one presented in Lemma 2.2 of \cite{Wa16} by Warnke (one may consult also \cite{BeDu22} by Bennett and Dudek).

\begin{lemma} [Freedman's inequality, \cite{Fr75}] \label{lemma:Freedman}
	Let $X_0, X_1, \ldots, X_n$ be a martingale with respect to a filtration $\mathcal{F}_0 \subseteq \mathcal{F}_1 \subseteq \ldots \subseteq \mathcal{F}_n$. Set $A_k = \max_{i \in [k]} (X_i - X_{i-1})$ and $W_k = \sum_{i=1}^{k} \var{X_i-X_{i-1}|\mathcal{F}_{i-1}}$. Then for every $\lambda>0$ and $W,A>0$ we have
	\[
	\Pr\big[\exists k \in [n] \quad X_k \geq X_0 + \lambda, W_k \leq W, A_k \leq A\big] \leq \exp\left\{-\frac{\lambda^2}{2W + 2A\lambda/3}\right\}.
	\]
\end{lemma}

Next, we present bounds for the values of $c_j$ and an upper bound for the function $\mu(\tilde{S})$, both introduced in Definition \ref{def:mu_delta_c}. These results will be referred to very often later on. They are derived using Stirling's approximation. 

\begin{lemma}[Stirling's approximation, \cite{Ro55}] \label{lemma:Stirling}
	Let $n \in \mathbb{N}$. Then
	\[
	\sqrt{2 \pi n} \left(\frac{n}{e}\right)^n \exp\left\{\frac{1}{12 n+1}\right\} <n!< \sqrt{2 \pi n} \left(\frac{n}{e}\right)^n \exp\left\{\frac{1}{12 n}\right\}.
	\]
\end{lemma}

\begin{lemma} \label{lemma:ct}
	For $j \geq 1$ let $c_j = \prod_{i=1}^{j}\frac{2i-1}{2i}$. Then
	\[
	\exp \left\{-\frac{1}{8j} - \frac{1}{4 \cdot 144 j^2}\right\} \cdot  \frac{1}{\sqrt{\pi j}} \leq c_j \leq \frac{1}{\sqrt{\pi j}.}
	\]
\end{lemma}
\begin{proof}
	By Stirling's approximation (Lemma \ref{lemma:Stirling}) we get
	\[
	c_j = \prod_{i=1}^{j}\frac{2i-1}{2i} = \frac{(2j)!}{2^{2j}(j!)^2} \leq \frac{1}{\sqrt{\pi j}}\exp\left\{\frac{1}{24 j}-\frac{2}{12j+1}\right\} \leq \frac{1}{\sqrt{\pi j}}.
	\]
	Analogously, since $\frac{1}{12 \cdot 2j +1} \geq \frac{1}{12\cdot 2j}-\frac{1}{144 (2j)^2}$,
	\[
	c_j \geq \frac{1}{\sqrt{\pi j}} \exp\left\{\frac{1}{24 j}-\frac{1}{4 \cdot 144 j^2}-\frac{2}{12 j}\right\} = \frac{1}{\sqrt{\pi j}} \exp\left\{-\frac{1}{8 j}-\frac{1}{4 \cdot 144 j^2}\right\}.
	\]
\end{proof}

\begin{lemma} \label{lemma:mu}
	Let $G_n^h$ be a preferential attachment graph and $T_{hn} = (\tilde{V},\tilde{E})$ its corresponding random tree. Let also $t \in [M]$ and $\tilde{S} \subseteq \tilde{V}$. Then
	\[
	\mu(\tilde{S}_t) \leq \sqrt{t} + \frac{1}{2}.
	\]
\end{lemma}
\begin{proof}[Proof of Lemma~\ref{lemma:mu}]
	By Lemma \ref{lemma:ct} we get
	\[
	\begin{split}
		\mu(\tilde{S}_t) & \leq \frac{\sqrt{\pi}}{2} \sum_{j=1}^{t} c_{j-1} \leq \frac{\sqrt{\pi}}{2} + \frac{1}{2} \sum_{j=2}^{t} \frac{1}{\sqrt{j-1}} \leq \frac{\sqrt{\pi}}{2} + \frac{1}{2} + \int_{1}^{t} \frac{1}{2 \sqrt{j}} \,dj \leq \sqrt{t} + \frac{1}{2}.
	\end{split}
	\]
\end{proof}

Lemmas \ref{lemma:only_delta}, \ref{lemma:Zi_first_part}, and \ref{lemma:Zi_second_part} are auxiliary calculations for expressions involving the volumes of subsets of $\tilde{V}$ (however the reader might not notice the connection with volumes at this point). They will be directly used in Section~\ref{sec:concentrations} in the proof of Theorem~\ref{thm:edges_concentration} stating the result on the concentration of $e_{G_n^h}(S)$ for $S \subseteq V$. For the proofs check the Appendix~\ref{app:A}.


\begin{lemma} \label{lemma:only_delta}
	Let $G_n^h$ be a preferential attachment graph and $T_{hn} = (\tilde{V},\tilde{E})$ its corresponding random tree. Fix $\tilde{S} \subseteq \tilde{V}$ and let $t_0 \in [M]$ be such that $t_0=t_0(n) \xrightarrow{n \rightarrow \infty} \infty$. Then
	\[
	\sum_{i=t_0+1}^{M} \frac{{\delta}_i}{2i-1} = \frac{\pi}{2} \sum_{i=1}^{M} ({\delta}_i c_{i-1})^2 + O(\ln{M}).
	\]
\end{lemma}

\begin{lemma} \label{lemma:Zi_first_part}
	Let $G_n^h$ be a preferential attachment graph and $T_{hn} = (\tilde{V},\tilde{E})$ its corresponding random tree. Fix $\tilde{S} \subseteq \tilde{V}$ and let $t_0 \in [M]$ be such that $t_0=t_0(n) \xrightarrow{n \rightarrow \infty} \infty$. Then
	\[
	\sum_{i=t_0+1}^{M} {\delta}_i \frac{2 \sqrt{i-1} \mu(\tilde{S}_{i-1})}{2i-1} = \frac{\pi}{2} \sum_{i=1}^{M} \left( {\delta}_i c_{i-1} \sum_{j=1}^{i-1} {\delta}_j c_{j-1} \right) + O(\ln{M}+t_0).
	\]
\end{lemma}

\begin{lemma} \label{lemma:Zi_second_part}
	Let $G_n^h$ be a preferential attachment graph and $T_{hn} = (\tilde{V},\tilde{E})$ its corresponding random tree. Fix $\tilde{S} \subseteq \tilde{V}$ and let $t_0 \in [M]$. Then for any constant $C>0$
	\[
	\sum_{i=t_0+1}^{M} {\delta}_i \frac{C(i-1)}{2i-1} \leq \frac{C}{2} (M-t_0).
	\]
\end{lemma}


\section{Edge density and volume results for {\boldmath{$G_n^h$}}} \label{sec:concentrations}

In this section we use martingale techniques to prove theorems~\ref{thm:vol_concentration}, \ref{thm:edges_concentration}, and \ref{thm:edges_between_concentration} stated in Section \ref{sec:VolDensity}, i.e., we derive concentration results for $\vol(S)$, $e(S)$, and $e(S, V\setminus S)$ for an arbitrary subset $S \subseteq V$ in $G_n^h$. 
A series of results will lead us to Corollary~\ref{cor:vol_conc_for_all_i} which implies, as a special case, Theorem~\ref{thm:vol_concentration}. We start with analyzing the volumes.

\begin{lemma} \label{lemma:Z^t_martingale}
	Let $G_n^h$ be a preferential attachment graph. Consider the process of constructing its corresponding random tree $T_{hn} = (\tilde{V},\tilde{E})$. Fix $\tilde{S} \subseteq \tilde{V}$ and for $t \in [M]$ let $Z_t = \vol_{T_{t}}(\tilde{S}_t)$. Set
	\[
	\hat{Z}_t = c_tZ_t - \sum_{j=1}^{t}{\delta}_j c_{j-1}
	\]
	(recall that ${\delta}_j$ and $c_j$ were introduced in Definition \ref{def:mu_delta_c}). Let $\mathcal{F}_t$ be a $\sigma$-algebra associated with all the events that happened till time $t$. Then $\hat{Z}_1, \hat{Z}_2, \ldots , \hat{Z}_M$ is a martingale with respect to the filtration $\mathcal{F}_1 \subseteq \ldots \subseteq \mathcal{F}_M$. Moreover, for $t \in [M-1]$
	\[
	|\hat{Z}_{t+1}-\hat{Z}_t| \leq \frac{2}{\sqrt{\pi t}} \quad \quad \textnormal{and} \quad \quad \var{(\hat{Z}_{t+1}-\hat{Z}_t)|\mathcal{F}_t} \leq \frac{1}{4 \pi(t+1)}.
	\]
\end{lemma}
\begin{remark*}
	The formula for $\hat{Z}_t$ was inspired by the martingale constructed in Lemma 4 of \cite{FrPePrRe20} by Frieze, Pra{\l}at, P{\'e}rez-Gim{\'e}nez, and Reiniger. The Frieze et al.'s martingale, in contrast to ours, consisted only of the first term, thus of $c_tZ_t = \left(\prod_{j=1}^{t}\frac{2j-1}{2j}\right) Z_t$, and therefore addressed only those sets of mini-vertices that formed compact intervals. By introducing the second term, i.e., by subtracting $ \sum_{j=1}^{t}{\delta}_j c_{j-1}$, we are able to handle all types of sets, including those scattered throughout the entire interval $[1,M]$.
\end{remark*}
\begin{proof}
	Let $t \in [M-1]$. Recall that when mini-vertex $(t+1)$ arrives, it may also connect to itself. Therefore, conditioned on $\mathcal{F}_{t}$,
	\begin{equation} \label{eq:Zt_cond_Ft}
		Z_{t+1} = \begin{cases} Z_{t} + \delta_{t+1} + 1 & \quad \textnormal {with probability} \quad \frac{Z_{t}+\delta_{t+1}}{2t+1} \\
			Z_{t} + \delta_{t+1} & \quad \textnormal{otherwise}.
		\end{cases}
	\end{equation}
	Additionally, since $c_{t} = c_{t+1} \cdot \frac{2t+2}{2t+1}$, we get
	\[
	\begin{split}
		\E[\hat{Z}_{t+1}|\mathcal{F}_{t}] & = \E\bigg[c_{t+1}Z_{t+1} - \sum_{j=1}^{t+1}\delta_j c_{j-1}\bigg|\mathcal{F}_{t}\bigg] 
		= c_{t+1}\left(Z_{t}+\delta_{t+1} + \frac{Z_{t}+\delta_{t+1}}{2t+1}\right)- \sum_{j=1}^{t+1}\delta_j c_{j-1} \\
		& = c_{t+1} \cdot \frac{2t+2}{2t+1}(Z_{t}+\delta_{t+1})- \sum_{j=1}^{t+1}\delta_j c_{j-1} 
		= c_{t}Z_{t} - \sum_{j=1}^{t}\delta_j c_{j-1} = \hat{Z}_{t},
	\end{split}
	\]
	thus $\hat{Z}_1, \ldots , \hat{Z}_M$ is a martingale with respect to the filtration $\mathcal{F}_1 \subseteq \ldots \subseteq \mathcal{F}_M$. Next, since $c_{t+1} = c_{t} \cdot \frac{2t+1}{2t+2}$
	\[
	\begin{split}
		|\hat{Z}_{t+1}-\hat{Z}_t| & = |c_{t+1}Z_{t+1}-c_{t}Z_{t}-\delta_{t+1}c_{t}| = c_{t}\left|\frac{2t+1}{2t+2}Z_{t+1} - Z_{t} - \delta_{t+1}\right| \\
		& = c_{t} \left|(Z_{t+1}-Z_{t})-\left(\frac{Z_{t+1}}{2t+2}+\delta_{t+1}\right)\right|.
	\end{split}
	\]
	Note that $(Z_{t+1} - Z_{t}) \in \{0,1,2\}$. Moreover, the volume of $\tilde{S}_{t+1}$ in $T_{t+1}$ may be at most $2t+2$. Thus $Z_{t+1} \leq 2t+2$, which also implies that $\left( \frac{Z_{t+1}}{2t+2}+\delta_{t+1}\right) \in [0,2]$. Now use Lemma \ref{lemma:ct} and the fact that $|a-b| \leq \max\{a,b\}$ for non-negative $a,b$ to get
	\[
	|\hat{Z}_{t+1}-\hat{Z}_t| \leq 2 c_{t} \leq \frac{2}{\sqrt{\pi t}}.
	\]
	By the fact that $\hat{Z}_1, \ldots , \hat{Z}_M$ is a martingale with respect to the filtration $\mathcal{F}_1 \subseteq \ldots \subseteq \mathcal{F}_M$, $c_t = \frac{2t+2}{2t+1} \cdot c_{t+1}$ and by (\ref{eq:Zt_cond_Ft}) we also have
	\[
	\begin{split}
		\var{(\hat{Z}_{t+1}-\hat{Z}_t) |\mathcal{F}_t}  & =
		\E[(\hat{Z}_{t+1}-\hat{Z}_t)^2|\mathcal{F}_t] 
		= \E[(c_{t+1}Z_{t+1} - c_tZ_t - {\delta}_{t+1}c_t)^2|\mathcal{F}_t] \\
		& = c_{t+1}^2 ~\E\left[\left(Z_{t+1} - \frac{2t+2}{2t+1}(Z_t + {\delta}_{t+1})\right)^2\Bigg|\mathcal{F}_t\right] \\
		& = c_{t+1}^2 \left(\left(Z_t + \delta_{t+1} + 1 - \frac{2t+2}{2t+1}(Z_t + {\delta}_{t+1})\right)^2 \frac{Z_{t}+\delta_{t+1}}{2t+1} \right.\\
		& \quad \quad \quad + \left. \left(Z_t + \delta_{t+1} - \frac{2t+2}{2t+1}(Z_t + {\delta}_{t+1})\right)^2  \left(1-\frac{Z_{t}+\delta_{t+1} }{2t+1}\right)\right) \\
		& = c_{t+1}^2 \left(\frac{Z_{t}+\delta_{t+1} }{2t+1}\left(1-\frac{Z_{t}+\delta_{t+1} }{2t+1}\right)\right) \leq \frac{c_{t+1}^2}{4} \leq \frac{1}{4 \pi (t+1)},
	\end{split}
	\]
	where the last inequalities follow from the fact that $\frac{Z_{t}+\delta_{t+1} }{2t+1} \in [0,1]$ and Lemma~\ref{lemma:ct}, respectively.
\end{proof}
The next proof utilizes the following well-known approximation.
\begin{lemma} [See \cite{BoWr71}] \label{lemma:harmonic}
	Let $n \in \mathbb{N}$ and $H_n = \sum_{k=1}^{n}\frac{1}{k}$. Then
	$
	H_n = \ln{n} + \gamma +\frac{1}{2n} - \alpha_n$,
	where $\gamma \approx 0.5772$ is known as Euler-Mascheroni constant and $0 \leq \alpha_n \leq 1/(8n^2)$.
\end{lemma}

\begin{theorem} \label{thm:vol_concentration_all}
	Let $G_n^h$ be a preferential attachment graph and $T_{hn} = (\tilde{V},\tilde{E})$ its corresponding random tree. Fix $\tilde{S} \subseteq \tilde{V}$ and let $t \in [M]$. 
	Then for every $\varepsilon >0$, for sufficiently large $t$ and for sufficiently large $n$ we get 
	\[
	\Pr\left[\left|\vol_{T_t}(\tilde{S}_t) - 2 \sqrt{t} \mu(\tilde{S}_t)\right| \geq (1+\varepsilon)g_{\mathcal{V}}(h) \frac{t}{\sqrt{h}}\right] \leq 2 \cdot 2^{-(1+{\varepsilon}/2)t/h},
	\]
	where $g_{\mathcal{V}}(h) = \frac{1}{6}\sqrt{2 \ln{2} \, (9 \ln{h}+8\ln{2})} + (2/3) \ln{2} + 2$.\\
	(Recall that $\mu(\tilde{S}_t)$ was introduced in Definition \ref{def:mu_delta_c}).
\end{theorem}
\begin{proof}
	Throughout the proof we refer to the process of constructing the random tree $T_{hn}$. For $t \in [M]$ let $\mathcal{F}_t$ be a $\sigma$-algebra associated with all the events that happened till time $t$. Fix $\varepsilon >0$, set $t_0 = \lfloor t/h \rfloor$ and for $j \in \{t_0, t_0+1, \ldots, t\}$ consider 
	\[
	\hat{Z}_j = c_jZ_j - \sum_{i=1}^{j}{\delta}_i c_{i-1},
	\]
	where $Z_j = \vol_{T_j}(\tilde{S}_j)$ and $c_i$'s and ${\delta}_i$'s are as in Definition \ref{def:mu_delta_c}. By Lemma \ref{lemma:Z^t_martingale} we know that $\hat{Z}_{t_0}, \ldots, \hat{Z}_t$ is a martingale with respect to the filtration $\mathcal{F}_{t_0} \subseteq \ldots \subseteq \mathcal{F}_t$ such that $|\hat{Z}_{j} - \hat{Z}_{j-1}| \leq \frac{2}{\sqrt{\pi (j-1)}}$ and $\var{(\hat{Z}_{j}-\hat{Z}_{j-1})|\mathcal{F}_{j-1}} \leq \frac{1}{4 \pi j}$. Therefore
	\[
	\max_{ j \in \{t_0+1, \ldots, t\}} (\hat{Z}_j - \hat{Z}_{j-1}) \leq \frac{2}{\sqrt{\pi (t_0-1)}} = \frac{2}{\sqrt{\pi (t/h)}} (1+ O(1/t))
	\]
	and, by Lemma \ref{lemma:harmonic},
	\[
	\begin{split}
		\sum_{j=t_0+1}^{t}   \var{(\hat{Z}_{j}-\hat{Z}_{j-1})|\mathcal{F}_{j-1}} 			 & \leq  \sum_{j=t_0+1}^{t}\frac{1}{4 \pi j} 
		\leq \frac{1}{4 \pi} \ln{(t/t_0)} + \frac{1}{4 \pi} \cdot \frac{1}{8\lfloor t/h \rfloor^2} \\
		& = \frac{1}{4 \pi} \ln{h} + O(1/t).
	\end{split}
	\]
	Applying Freedman's inequality (Lemma \ref{lemma:Freedman}) to $\hat{Z}_{t_0}, \ldots, \hat{Z}_t$ with $A = \frac{2}{\sqrt{\pi (t/h)}}(1+O(1/t))$, $W = \frac{1}{4 \pi} \ln{h} + O(1/t)$ and
	$\lambda = \frac{(1+\varepsilon)\bar{g}_{\mathcal{V}}(h)}{\sqrt{\pi}} \sqrt{t/h}$, where $\bar{g}_{\mathcal{V}}(h) = g_{\mathcal{V}}(h)-2$, we get
	\[
	\begin{split}
		\Pr\left[\hat{Z}_t \vphantom{\frac{(1+\varepsilon)\bar{g}_{\mathcal{V}}(h)}{\sqrt{\pi}}}\right.& \left. \geq \hat{Z}_{t_0} + \frac{(1+\varepsilon)\bar{g}_{\mathcal{V}}(h)}{\sqrt{\pi}} \sqrt{t/h} \right] \\
		&\leq \exp\left\{-\frac{\frac{(1+\varepsilon)^2\bar{g}_{\mathcal{V}}(h)^2}{\pi} \frac{t}{h}}{2 \cdot \frac{1}{4 \pi} \ln{h} +O(1/t) + \frac{2}{3} \cdot \frac{2}{\pi} (1+\varepsilon)\bar{g}_{\mathcal{V}}(h)(1+O(1/t))}\right\} \\
		&\leq \exp\left\{-\frac{(1+\varepsilon/2)^2\bar{g}_{\mathcal{V}}(h)^2 }{\frac{1}{2} \ln{h} + \frac{4}{3} (1+\varepsilon/2)\bar{g}_{\mathcal{V}}(h)} \cdot \frac{t}{h}\right\},
	\end{split}
	\]
	where the last inequality holds for sufficiently large $t$ and $n$. One can verify that
	\[
	\frac{(1+\varepsilon/2)^2\bar{g}_{\mathcal{V}}(h)^2 }{\frac{1}{2} \ln{h} + \frac{4}{3} (1+\varepsilon/2)\bar{g}_{\mathcal{V}}(h)} \geq (1+{\varepsilon}/2) \ln{2}
	\]
	thus for sufficiently large $t$ and $n$ we get
	\begin{equation} \label{eq:prob_hatZ+}
		\Pr\left[\hat{Z}_t \geq \hat{Z}_{t_0} + \frac{(1+\varepsilon)\bar{g}_{\mathcal{V}}(h)}{\sqrt{\pi}}\sqrt{t/h}\right] \leq 2^{-(1+{\varepsilon}/2)t/h}.
	\end{equation}
	
	Let us now analyze the complementary event $\left\{\hat{Z}_t \leq \hat{Z}_{t_0} + \frac{(1+\varepsilon)\bar{g}_{\mathcal{V}}(h)}{\sqrt{\pi}}\sqrt{t/h}\right\}$. It is equivalent to
	\[
	\left\{c_t Z_t - \sum_{i=1}^{t}{\delta}_i c_{i-1} \leq c_{t_0} Z_{t_0} - \sum_{i=1}^{t_0}{\delta}_i c_{i-1} + \frac{(1+\varepsilon)\bar{g}_{\mathcal{V}}(h)}{\sqrt{\pi}}\sqrt{t/h}\right\}
	\]
	which is
	\[
	\left\{Z_t - \frac{1}{c_t}\sum_{i=1}^{t}{\delta}_i c_{i-1} \leq \frac{c_{t_0}}{c_t} Z_{t_0} - \frac{1}{c_t}\sum_{i=1}^{t_0}{\delta}_i c_{i-1} + \frac{1}{c_t} \frac{(1+\varepsilon)\bar{g}_{\mathcal{V}}(h)}{\sqrt{\pi}}\sqrt{t/h}\right\}.
	\]
	By the definition of $\mu(\tilde{S}_t)$, Lemma \ref{lemma:ct} and Lemma \ref{lemma:mu} we have
	\begin{equation} \label{eq:bound+_mu}
		\begin{split}
			\frac{1}{c_t} \sum_{i=1}^{t}{\delta}_i c_{i-1} & = \frac{2}{\sqrt{\pi} c_t} \mu(\tilde{S}_t) \leq 2 \sqrt{t} \mu(\tilde{S}_t) e^{\frac{1}{8t}+\frac{1}{4 \cdot 144 t^2}} \\
			& = 2 \sqrt{t} \mu(\tilde{S}_t) (1+O(1/t)) = 2 \sqrt{t} \mu(\tilde{S}_t) + O(1).
		\end{split}
	\end{equation}
	Note that $Z_{t_0} \leq 2 t_0$, therefore by Lemma \ref{lemma:ct}
	\begin{equation} \label{eq:bound+_proces_begin}
		\frac{c_{t_0}}{c_t} Z_{t_0} \leq 2 \sqrt{t} \sqrt{t_0} (1+O(1/t)) = \frac{2 t}{\sqrt{h}} + O(1),
	\end{equation}
	and, again by Lemma \ref{lemma:ct},
	\begin{equation} \label{eq:bound+_x_Azuma}
		\begin{split}
			\frac{1}{c_t} \frac{(1+\varepsilon)\bar{g}_{\mathcal{V}}(h)}{\sqrt{\pi}}\sqrt{t/h} & \leq (1+\varepsilon)\bar{g}_{\mathcal{V}}(h)\frac{t}{\sqrt{h}}~(1+O(1/t)) 
			= (1+\varepsilon)\bar{g}_{\mathcal{V}}(h)\frac{t}{\sqrt{h}} + O(1).
		\end{split}
	\end{equation}
	Thus by (\ref{eq:bound+_mu}), (\ref{eq:bound+_proces_begin}) and (\ref{eq:bound+_x_Azuma}) we get that the event $\left\{\hat{Z}_t \leq \hat{Z}_{t_0} + \frac{(1+\varepsilon)\bar{g}_{\mathcal{V}}(h)}{\sqrt{\pi}}\sqrt{t/h}\right\}$ implies
	\[
	\left\{Z_t - 2 \sqrt{t}\mu(\tilde{S}_t) \leq \frac{2t}{\sqrt{h}} + (1+\varepsilon)\bar{g}_{\mathcal{V}}(h)\frac{t}{\sqrt{h}} + O(1)\right\}
	\]
	which, by (\ref{eq:prob_hatZ+}), for sufficiently large $t$ and $n$, gives (recall that $\bar{g}_{\mathcal{V}}(h) + 2 =  g_{\mathcal{V}}(h)$)
	\begin{equation} \label{eq:P_bound+}
		\Pr\left[\vol_{T_t}(\tilde{S}_t) - 2 \sqrt{t}\mu(\tilde{S}_t) \geq (1+\varepsilon)g_{\mathcal{V}}(h)\frac{t}{\sqrt{h}} \right] \leq 2^{-(1+{\varepsilon}/2)t/h}.
	\end{equation}
	To get the opposite bound we repeat the reasoning for the martingale $-\hat{Z}_{t_0}, \ldots, -\hat{Z}_t$ (check the Appendix~\ref{app:B} for the details).
\end{proof}

\begin{corollary} \label{cor:vol_conc_for_all_i}
	Let $G_n^h = (V,E)$ be a preferential attachment graph and $T_{hn} = (\tilde{V},\tilde{E})$ its corresponding random tree.
	Then for every $\varepsilon >0$ 
	we have 
	\[ 
	\begin{split}
		\Pr \left[\forall i \in \{\lfloor \log_2{n} \rfloor, \ldots, n\} \, \forall S \subseteq V 
		\left|\vol_{G_i^h}(S_i) - 2 \sqrt{hi} \, \mu(\tilde{S}_{hi})\right| \leq  (1+\varepsilon)g_{\mathcal{V}}(h) 
		\frac{hi}{\sqrt{h}}
		\right] = 1 - o(1), 
	\end{split}
	\]
	where $g_{\mathcal{V}}(h) = \frac{1}{6}\sqrt{2 \ln{2} \, (9 \ln{h}+8\ln{2})} + (2/3) \ln{2} + 2$.
\end{corollary}

\begin{proof}
	Fix $\varepsilon>0$. Recall that $\vol_{G_i^h}(S_i)= \vol_{T_{hi}}(\tilde{S}_{hi})$. For $S \subseteq V$ and $i \in \{\lfloor \log_2{n} \rfloor, \ldots, n\}$ define the event $\mathcal{E}_{S,i}$ as follows 
	\[
	\mathcal{E}_{S,i} = \Bigl\{\left|\vol_{T_{hi}}(\tilde{S}_{hi}) - 2 \sqrt{hi} \, \mu(\tilde{S}_{hi})\right| \leq  (1+\varepsilon)g_{\mathcal{V}}(h) \frac{hi}{\sqrt{h}} \Bigr\}.
	\]
	For $i \in \{\lfloor \log_2{n} \rfloor, \ldots, n\}$, by Theorem \ref{thm:vol_concentration_all} and the union bound, for sufficiently large $n$ we have
	\[
	\Pr\left[\exists S \subseteq V \,\, \mathcal{E}_{S,i}^C \right] \leq 2^{i} \cdot 2 \cdot 2^{-(1+\varepsilon/2)i} = 2 \cdot 2^{-(\varepsilon/2) i}.
	\]
	Indeed, note that $i$ iterates over the vertices of $G_n^h$ and at time $i$ there are $2^i$ possible configurations for $S_i$, thus also for $\tilde{S}_{hi}$. Next, again by the union bound, for sufficiently large $n$ we get
	\[
	\begin{split}
		\Pr\left[\exists i \in \{\lfloor \log_2{n} \rfloor, \ldots, n\} \, \, \exists S \subseteq V \,\, \mathcal{E}_{S,i}^C \right] & \leq \sum_{i=\lfloor \log_2{n} \rfloor}^{n}  2 \cdot 2^{-(\varepsilon/2) i} \\
		& \leq \frac{2 \cdot 2^{-\varepsilon \lfloor \log_2{n} \rfloor}}{1-2^{-\varepsilon}} \sim \frac{2}{(1-2^{-\varepsilon}) n^{\varepsilon}},
	\end{split}
	\]
	which implies
	\[
	\Pr\left[\forall i \in \{\lfloor \log_2{n} \rfloor, \ldots, n\} \, \, \forall S \subseteq V \,\, \mathcal{E}_{S,i} \right] = 1-o(1).
	\]
\end{proof}

Note that considering only $i=n$ in Corollary \ref{cor:vol_conc_for_all_i} we get the statement of Theorem \ref{thm:vol_concentration}, which finishes its proof.

Now we will again use martingale inequalities which, together with concentration results for volumes from Corollary~\ref{cor:vol_conc_for_all_i}, will lead to the proof of Theorem~\ref{thm:edges_concentration}. 	Consider the process of constructing the random tree $T_{hn} =(\tilde{V}, \tilde{E})$. Let $\tilde{S} \subseteq \tilde{V}$ and $j \in \{1,\ldots,M\}$. The result stated in Corollary \ref{cor:vol_conc_for_all_i} gives the concentration of the volumes of $\tilde{S}_j$ at time $j$ only for $j=hi$, where $i \in \{\lfloor \log_2{n} \rfloor, \ldots, n\}$. In particular, it says nothing about the concentration of the volumes of the sets $\tilde{S}_{hi+k}$, where $k \in \{1,2,\ldots h-1\}$, at time $hi+k$. Such intermediate concentrations will be needed to prove Theorem \ref{thm:edges_concentration}, i.e., to draw a conclusion about the concentration of the number of edges within~$\tilde{S}$. We derive those intermediate concentrations in Lemma~\ref{lemma:Zis_all_concentrated}. 

\begin{lemma} \label{lemma:Zis_all_concentrated}
	Let $G_n^h = (V,E)$ be a preferential attachment graph and $T_{hn} = (\tilde{V},\tilde{E})$ its corresponding random tree. Then for every $\varepsilon >0$ 
	we have 
	\[
	\begin{split}
		\Pr \left[ \forall t  \in \{\lfloor \log_2{n} \rfloor h , \ldots, M\} \, \forall S \subseteq V 
		\left|\vol_{T_{t}}(\tilde{S}_{t}) - 2 \sqrt{t} \, \mu(\tilde{S}_{t})\right| \leq  (1+\varepsilon)g_{\mathcal{V}}(h) \frac{t}{\sqrt{h}}\right] = 1 - o(1), 
	\end{split}
	\]
	where $g_{\mathcal{V}}(h) = \frac{1}{6}\sqrt{2 \ln{2} \, (9 \ln{h}+8\ln{2})} + (2/3) \ln{2} + 2$.
\end{lemma}

\begin{proof} 
	Fix $\varepsilon > 0$. Define the events $\mathcal{H}$ and $\mathcal{V}$ as follows
	
	\[ 
	\begin{split}
		\mathcal{H} & = 	\left\{  \forall t \in \{\lfloor \log_2{n} \rfloor h , \ldots, M\} \, \forall S \subseteq V  \left|\vol_{T_{t}}(\tilde{S}_{t}) - 2 \sqrt{t} \, \mu(\tilde{S}_{t})\right| \leq  (1+\varepsilon)g_{\mathcal{V}}(h) \frac{t}{\sqrt{h}}\right\}\\
		\mathcal{V} & = \left\{ \vphantom{\frac{C}{\sqrt{h}}} \right. \forall j  = ih  \textnormal{ such that } i \in \{\lfloor \log_2{n} \rfloor, \ldots, n\} \, \forall S \subseteq V \\
		& \quad \quad \quad \quad \quad \quad
		\left. \left|\vol_{T_j}(\tilde{S}_j) - 2 \sqrt{j} \, \mu(\tilde{S}_{j})\right| \leq  (1+\varepsilon/2)g_{\mathcal{V}}(h) \frac{j}{\sqrt{h}}\right\}.
	\end{split}
	\]

	By Corollary \ref{cor:vol_conc_for_all_i} we know that $\Pr[\mathcal{V}] = 1-o(1)$ thus it is enough to show that the event $\mathcal{V}$ implies the event $\mathcal{H}$.
	
	Assume that $\mathcal{V}$ holds. For $t \in [M]$ and $S \subseteq V$ let $Z_t = \vol_{T_t}(\tilde{S}_t)$. Consider all $j = ih$ where $i \in \{\lfloor \log_2{n} \rfloor, \ldots, n-1\}$ and let $k \in \{1,2,\ldots h-1\}$. 
	By the fact that $\sqrt{1+k/j} = 1+O(1/j)$, ${\delta}_{\ell} \in \{0,1\}$, Lemma \ref{lemma:ct} and Lemma \ref{lemma:mu} we get
	\[
	\begin{split}
		2 \sqrt{j+k} \, \mu(\tilde{S}_{j+k}) & = 2 \sqrt{j} \sqrt{1+k/j} \left(\mu(\tilde{S}_{j}) + \sum_{\ell = j+1}^{j+k}{{\delta}_{\ell} c_{\ell-1}}\right) \\
		& \leq 2 \sqrt{j} \left(1+ O(1/j)\right) \left( \mu(\tilde{S}_{j}) + \sum_{\ell = j+1}^{j+k} \frac{1}{\sqrt{\pi (\ell -1)}}\right) \\
		& = (2 \sqrt{j} + O(1/\sqrt{j})) \left(\mu(\tilde{S}_j) + O(1/\sqrt{j})\right) 
		= 2 \sqrt{j} \, \mu(\tilde{S}_j) + O(1).
	\end{split}
	\]
	Therefore, if $\mathcal{V}$ holds, then for all $j = ih$ where $i \in \{\lfloor \log_2{n} \rfloor, \ldots, n-1\}$, for all $S \subseteq V$, and for all $k \in \{1,2,\ldots,h-1\}$, for sufficiently large $n$ on one hand, 
	\[
	\begin{split}
		Z_{j+k} & \geq Z_j \geq 2 \sqrt{j} \, \mu(\tilde{S}_j) - (1+\varepsilon/2)g_{\mathcal{V}}(h)\frac{j}{\sqrt{h}} \\
		& \geq 2 \sqrt{j+k} \, \mu(\tilde{S}_{j+k}) - O(1) - (1+\varepsilon/2)g_{\mathcal{V}}(h)\frac{j+k}{\sqrt{h}} \\
		& \geq 2 \sqrt{j+k} \, \mu(\tilde{S}_{j+k}) - (1+\varepsilon)g_{\mathcal{V}}(h)\frac{j+k}{\sqrt{h}}
	\end{split}
	\]
	and on the other hand
	\[
	\begin{split}
		Z_{j+k} & \leq Z_j +2k \leq 2 \sqrt{j} \, \mu(\tilde{S}_j) + (1+\varepsilon/2)g_{\mathcal{V}}(h)\frac{j}{\sqrt{h}} + 2h \\
		& \leq 2 \sqrt{j+k} \, \mu(\tilde{S}_{j+k}) + (1+\varepsilon)g_{\mathcal{V}}(h)\frac{j+k}{\sqrt{h}}.
	\end{split}
	\]
	Thus for sufficiently large $n$ the event $\mathcal{V}$ implies the event $\mathcal{H}$, therefore $\Pr[\mathcal{V}]=1-o(1)$ implies $\Pr[\mathcal{H}]=1-o(1)$ and the proof is finished.
\end{proof}

Now, we move on to the concentration of the number of edges within subsets of $V$.

\begin{lemma} \label{lemma:X^t_martingale}
	Let $G_n^h$ be a preferential attachment graph. Consider the process of constructing its corresponding random tree $T_{hn} = (\tilde{V},\tilde{E})$.	Fix $\tilde{S} \subseteq \tilde{V}$ and for $t \in [M]$ let $X_t = e(\tilde{S}_t)$ and $\mathcal{F}_t$ be a $\sigma$-algebra associated with all the events that happened till time $t$. For $j \in \{2,\ldots,M\}$ set $D_j = \E[X_{j}-X_{j-1}|\mathcal{F}_{j-1}]$ and define
	\[
	\hat{X}_t = X_t - \sum_{j=2}^{t}D_j.
	\]
	Then $\hat{X}_1, \hat{X}_2, \ldots , \hat{X}_M$ is a martingale with respect to the filtration $\mathcal{F}_1 \subseteq \ldots \subseteq \mathcal{F}_M$. Moreover, for $t \in \{2,\ldots,M\}$
	\[
	|\hat{X}_{t}-\hat{X}_{t-1}| \leq {\delta}_{t}.
	\]
\end{lemma}

\begin{proof}
	Let $t \in \{2,\ldots,M\}$. Note that
	\[
	\begin{split}
		\E[\hat{X}_t  - \hat{X}_{t-1} | \mathcal{F}_{t-1}] & = \E[{X}_t - {X}_{t-1} -D_t| \mathcal{F}_{t-1}] \\
		& =  \E[{X}_t - {X}_{t-1}| \mathcal{F}_{t-1}] - \E[\E[X_{t}-X_{t-1}|\mathcal{F}_{t-1}]| \mathcal{F}_{t-1}] \\
		& = \E[{X}_t - {X}_{t-1}| \mathcal{F}_{t-1}] - \E[{X}_t - {X}_{t-1}| \mathcal{F}_{t-1}] = 0,
	\end{split}
	\]
	thus $\hat{X}_1, \ldots, \hat{X}_M$ is a martingale with respect to the filtration $\mathcal{F}_1 \subseteq \ldots \subseteq \mathcal{F}_M$.
	
	Now, for $t \in [M]$ let $Z_t = \vol_{T_t}(\tilde{S}_t)$. Recall that when the mini-vertex $t$ arrives it may also connect to itself thus for $t \in \{2,\ldots, M\}$ we have
	\[
	D_t = \E[X_t - X_{t-1}| \mathcal{F}_{t-1}] = {\delta}_t \frac{Z_{t-1}+1}{2t-1}.
	\]
	Therefore 
	\[
	\begin{split}
		|\hat{X}_{t}-\hat{X}_{t-1}| & = |X_{t} - X_{t-1} - D_t| = \left|X_{t} - X_{t-1} - {\delta}_t \frac{Z_{t-1}+1}{2t-1}\right| \\
		& \leq \max\left\{X_t-X_{t-1}, {\delta}_t \frac{Z_{t-1}+1}{2t-1}\right\} \leq {\delta}_t,
	\end{split}
	\]
	where we used the fact that $0 \leq X_t - X_{t-1} \leq {\delta}_t$, $Z_{t-1} \leq 2t-2$ thus $0 \leq {\delta}_t \frac{Z_{t-1}}{2t-1} \leq {\delta}_t$ and  $|a-b| \leq \max\{a,b\}$ for non-negative $a,b$.
\end{proof}

\begin{lemma} \label{lemma:XM_as_Zis}
	Let $G_n^h$ be a preferential attachment graph and $T_{hn} = (\tilde{V},\tilde{E})$ its corresponding random tree. For $t \in [M]$ and $S \subseteq V$ let $Z_t = \vol_{T_{t}}(\tilde{S}_t)$. 
	Then for $t_0 \in [M]$ divisible by $h$ and for every $\varepsilon>0$ we have
	\[
	\Pr\left[ \forall S \subseteq V \,
	\left|e(S) - \left(e(S_{t_0/h}) + \sum_{i=t_0+1}^{M} {\delta}_i \frac{Z_{i-1}+1}{2i-1} \right)\right| \leq {B}_{\varepsilon} \frac{M}{\sqrt{h}}
	\right] = 1 - o(1),
	\]
	where ${B}_{\varepsilon} = \sqrt{(1+\varepsilon)2\ln{2}}$.
\end{lemma}
\begin{proof}
	Throughout the proof we again refer to the process of constructing the random tree $T_{hn}$. For $t \in [M]$ let $\mathcal{F}_t$ be a $\sigma$-algebra associated with all the events that happened till time $t$. For $S \subseteq V$ let $X_t = e_{T_t}(\tilde{S}_t)$ and $D_t = \E[X_{t}-X_{t-1}|\mathcal{F}_{t-1}]$. Fix $\varepsilon>0$ and for $j \in \{t_0, t_0+1, \ldots, M\}$ and $S \subseteq V$ consider
	\[
	\hat{X}_j = X_j - \sum_{i=2}^{j} D_i.
	\]
	By Lemma \ref{lemma:X^t_martingale} we know that $\hat{X}_{t_0}, \ldots, \hat{X}_M$ is a martingale with respect to the filtration $\mathcal{F}_{t_0} \subseteq \ldots \subseteq \mathcal{F}_M$ such that $|\hat{X}_{j} - \hat{X}_{j-1}| \leq {\delta}_j$. Moreover 
	\[
	\sum_{j=t_0}^{M} {\delta}_j^2 = \sum_{j=t_0}^{M} {\delta}_j \leq M. 
	\]
	Thus applying Azuma-Hoeffding inequality (Lemma \ref{lemma:Azuma}) to $\hat{X}_{t_0}, \ldots, \hat{X}_M$ with $b_j = {\delta}_j$ and $x = {B}_{\varepsilon} \frac{M}{\sqrt{h}}$, where ${B}_{\varepsilon} = \sqrt{(1+\varepsilon)2\ln{2}}$ we get
	\begin{equation} \label{eq:hat_XM_concentr+}
		\begin{split}
			\Pr\left[\hat{X}_M \geq \hat{X}_{t_0} + {B}_{\varepsilon} \frac{M}{\sqrt{h}}\right] & \leq \exp\left\{-\frac{(1+\varepsilon) 2\ln{2} \cdot M^2/h}{2 M}\right\} = 2^{-(1+\varepsilon)n}.
		\end{split}
	\end{equation}
	Let us now analyze the event $\left\{\hat{X}_M \geq \hat{X}_{t_0} + {B}_{\varepsilon} \frac{M}{\sqrt{h}}\right\}$. It is equivalent to 
	\[
	\left\{X_M - \sum_{i=2}^{M}D_i \geq X_{t_0} - \sum_{i=2}^{t_0} D_i + {B}_{\varepsilon} \frac{M}{\sqrt{h}}\right\}
	\]
	which is
	\[
	\left\{X_M \geq X_{t_0} + \sum_{i=t_0+1}^{M} D_i + {B}_{\varepsilon} \frac{M}{\sqrt{h}}\right\}
	\]
	and, by the definition of $D_i$ (check the proof of Lemma \ref{lemma:X^t_martingale}),
	\[ 
	\left\{X_M \geq X_{t_0} + \sum_{i=t_0+1}^{M} {\delta}_i\frac{Z_{i-1}+1}{2i-1} + {B}_{\varepsilon} \frac{M}{\sqrt{h}}\right\}.
	\]
	Thus, by (\ref{eq:hat_XM_concentr+}), we get
	\begin{equation} \label{eq:XM_lower}
		\Pr\left[X_M - \left( X_{t_0} + \sum_{i=t_0+1}^{M} {\delta}_i\frac{Z_{i-1}+1}{2i-1} \right)  \geq {B}_{\varepsilon} \frac{M}{\sqrt{h}}\right] \leq 2^{-(1+\varepsilon)n}.
	\end{equation}
	Acting analogously for the martingale $-\hat{X}_{t_0}, \ldots, -\hat{X}_M$ we get the opposite bound (check the Appendix~\ref{app:B} for the details).
\end{proof}

We are ready to prove Theorem \ref{thm:edges_concentration}.
\begin{proof}[Proof of Theorem~\ref{thm:edges_concentration}]
	Throughout the proof we again refer to the process of constructing the random tree $T_{hn}$. Fix $\varepsilon >0$. For $t \in [M]$ and $S \subseteq V$ let $X_t = e_{T_t}(\tilde{S}_t)$ and $Z_t = \vol_{T_{t}}(\tilde{S}_t)$.
	For $t_0 = h\lfloor \log_2{n} \rfloor$ we define the events $\mathcal{H}$, $\mathcal{E}$ and $\mathcal{V}$ as follows 
	\[
	\begin{split}
		\mathcal{H} = \left\{ \vphantom{\frac{C}{\sqrt{h}}}\forall S \right. & \subseteq V 
		\left. \left| e(S) - \mu(\tilde{S})^2 \right| \leq (1+\varepsilon)g_{\mathcal{E}}(h) \frac{M}{\sqrt{h}} \right\},\\
		\mathcal{E} = \left\{ \vphantom{\frac{C}{\sqrt{h}}}\forall S \right. & \subseteq V 
		\left. \left|e(S) - \left(e(S_{t_0/h}) + \sum_{i=t_0+1}^{M} {\delta}_i \frac{Z_{i-1}+1}{2i-1} \right)\right| \leq {B}_{\varepsilon} \frac{M}{\sqrt{h}} \right\},\\
		\mathcal{V} = \left\{  \vphantom{\frac{C}{\sqrt{h}}} \forall t \right.  & \in \{t_0, \ldots, M\} \, \forall S \subseteq V 
		\left. |Z_t - 2 \sqrt{t} \, \mu(\tilde{S}_{t})| \leq  (1+\varepsilon)g_{\mathcal{V}}(h) \frac{t}{\sqrt{h}}\right\},
	\end{split}
	\]
	where ${B}_{\varepsilon} = \sqrt{(1+\varepsilon)2\ln{2}}$. By lemmas \ref{lemma:Zis_all_concentrated} and \ref{lemma:XM_as_Zis} we know that $\Pr[\mathcal{E} \cap \mathcal{V}] = 1-o(1)$. Thus it is enough to show that the event $\mathcal{E} \cap \mathcal{V}$ implies the event $\mathcal{H}$.
	
	Assume that $\mathcal{E} \cap \mathcal{V}$ holds. By lemmas \ref{lemma:only_delta}, \ref{lemma:Zi_first_part}, and \ref{lemma:Zi_second_part}, for any $S \subseteq V$, as $t_0=O(\ln M)$, we can write
	\[
	\begin{split}
		\sum_{i=t_0+1}^{M} 
		{\delta}_i & \frac{Z_{i-1}+1}{2i-1} 
		\leq \sum_{i=t_0+1}^{M} \frac{{\delta}_i}{2i-1} + \sum_{i=t_0+1}^{M} {\delta}_i \frac{2 \sqrt{i-1} \mu(\tilde{S}_{i-1})}{2i-1} + \sum_{i=t_0+1}^{M} {\delta}_i \frac{(1+\varepsilon)g_{\mathcal{V}}(h) \frac{i-1}{\sqrt{h}}}{2i-1} 
		\\
		\leq & ~\frac{\pi}{2} \sum_{i=1}^{M} ({\delta}_i c_{i-1})^2 + \frac{\pi}{2} \sum_{i=1}^{M} \left( {\delta}_i c_{i-1} \sum_{j=1}^{i-1} {\delta}_j c_{j-1} \right) 
		+ \frac{(1+\varepsilon)g_{\mathcal{V}}(h)}{2} \frac{M-t_0}{\sqrt{h}} + O(\ln{M}) \\
		= &\left(\frac{\sqrt{\pi}}{2} \right)^2\left( \sum_{i=1}^{M} ({\delta}_i c_{i-1})^2 + 2 \sum_{i=1}^{M} \left( {\delta}_i c_{i-1} \sum_{j=1}^{i-1} {\delta}_j c_{j-1} \right) \right) + \frac{\pi}{4} \sum_{i=1}^{M} ({\delta}_i c_{i-1})^2 \\
		& + \frac{(1+\varepsilon)g_{\mathcal{V}}(h)}{2} \frac{M-t_0}{\sqrt{h}} + O(\ln{M}) \\
		\leq & ~\mu(\tilde{S}_M)^2 + \frac{\pi}{4} \sum_{i=1}^{M} \frac{1}{\pi(i-1)} + \frac{(1+\varepsilon)g_{\mathcal{V}}(h)}{2} \frac{M-t_0}{\sqrt{h}} + O(\ln{M}) \\
		= & ~\mu(\tilde{S}_M)^2 + \frac{(1+\varepsilon)g_{\mathcal{V}}(h)}{2} \frac{M-t_0}{\sqrt{h}} + O(\ln{M}),
	\end{split}
	\]
	where the last inequality follows from Lemma \ref{lemma:ct}, the fact that ${\delta}_i \in \{0,1\}$, and the fact that 
	\[
	\mu(\tilde{S}_M)^2 = \left( \frac{\sqrt{\pi}}{2} \sum_{i=1}^{M} {\delta}_i c_{i-1} \right)^2 =  \left(\frac{\sqrt{\pi}}{2} \right)^2\left( \sum_{i=1}^{M} ({\delta}_i c_{i-1})^2 + 2 \sum_{i=1}^{M} \left( {\delta}_i c_{i-1} \sum_{j=1}^{i-1} {\delta}_j c_{j-1} \right) \right).
	\]
	Analogously, again by lemmas \ref{lemma:only_delta}, \ref{lemma:Zi_first_part}, and \ref{lemma:Zi_second_part}, for any $S \subseteq V$ we can get
	\[
	\begin{split}
		\sum_{i=t_0+1}^{M}  {\delta}_i\frac{Z_{i-1}+1}{2i-1}   & \geq \sum_{i=t_0+1}^{M} \frac{{\delta}_i}{2i-1} + \sum_{i=t_0+1}^{M} {\delta}_i \frac{2 \sqrt{i-1} \mu(\tilde{S}_{i-1})}{2i-1} - \sum_{i=t_0+1}^{M} {\delta}_i \frac{(1+\varepsilon)g_{\mathcal{V}}(h) \frac{i-1}{\sqrt{h}}}{2i-1} \\
		& \geq \mu(\tilde{S}_M)^2 - \frac{(1+\varepsilon)g_{\mathcal{V}}(h)}{2} \frac{M-t_0}{\sqrt{h}} - O(\ln{M}).
	\end{split}
	\]
	Note that for any $S \subseteq V$ we have $e(S_{t_0/h}) \leq t_0$ and recall that $t_0 = h \lfloor \log_2{n} \rfloor$. Thus for all $S \subseteq V$, for sufficiently large $n$ we may write (recall that $\mathcal{E} \cap \mathcal{V}$ holds)
	\[
	\begin{split}
		e(S) & \leq t_0  + {B}_{\varepsilon} \frac{M}{\sqrt{h}} + \mu(\tilde{S})^2 + \frac{(1+\varepsilon)g_{\mathcal{V}}(h)}{2} \frac{M-t_0}{\sqrt{h}} + O(\ln{M}) \\
		& \leq \mu(\tilde{S})^2 + (1+\varepsilon) \sqrt{2 \ln{2}}\frac{M}{\sqrt{h}} + \frac{(1+\varepsilon)g_{\mathcal{V}}(h)}{2} \frac{M}{\sqrt{h}} 
		= \mu(\tilde{S})^2 + (1+\varepsilon)g_{\mathcal{E}}(h) \frac{M}{\sqrt{h}},
	\end{split}
	\]
	where the term $O(\ln{M}) +t_0$ vanishes after the second inequality as the factor $\sqrt{1+\varepsilon}$ from $B_{\varepsilon}$ is replaced by $(1+\varepsilon)$. Analogously we get
	\[
	\begin{split}
		e(S) & \geq  - {B}_{\varepsilon} \frac{M}{\sqrt{h}} + \mu(\tilde{S})^2 - \frac{(1+\varepsilon)g_{\mathcal{V}}(h)}{2} \frac{M-t_0}{\sqrt{h}} - O(\ln{M}) 
		\geq \mu(\tilde{S})^2 - (1+\varepsilon)g_{\mathcal{E}}(h) \frac{M}{\sqrt{h}}.
	\end{split}
	\]
	This means that for sufficiently large $n$, the event $\mathcal{E} \cap \mathcal{V}$ implies the event $\mathcal{H}$. Thus $\Pr[\mathcal{E} \cap \mathcal{V}] = 1 - o(1)$ implies $\Pr[\mathcal{H}] = 1-o(1)$, and the proof is finished.
\end{proof}
Mimicking the reasoning from the proofs of Lemma \ref{lemma:X^t_martingale}, Lemma \ref{lemma:XM_as_Zis}, and Theorem~\ref{thm:edges_concentration} one can prove Theorem~\ref{thm:edges_between_concentration}. We leave it without the proof details as this would be very repetitive.


\section{Modularity of {\boldmath{$G_n^h$}} vanishes with {\boldmath{$h$}}} \label{sec:modularity}

Recall that the main result of the paper (Theorem \ref{thm:main})
states that the modularity of a preferential attachment graph $G_n^h$ is with high probability upper bounded by a function tending to $0$ with $h$ tending to infinity. The whole current section is devoted to its proof. 

The first step in the proof of Theorem \ref{thm:main} follows from an interesting general result on modularity by Dinh and Thai.

\begin{lemma}[\cite{DiTh11}, Lemma 1] \label{lemma:only2parts}
	Let $G$ be a graph with at least one edge and let $k \in \mathbb{N} \setminus \{1\}$. Then
	\[
	\modul(G) \leq \frac{k}{k-1} \max_{\mathcal{A}: |\mathcal{A}| \leq k} \modul_{\mathcal{A}}(G).
	\]
	In particular,
	\[
	\modul(G) \leq 2  \max_{\mathcal{A}: |\mathcal{A}| \leq 2} \modul_{\mathcal{A}}(G).
	\]
\end{lemma}

\begin{corollary} \label{cor:mod_upperbound}
	Let $G=(V,E)$ be a graph with at least one edge. Then 
	\[
	\modul(G) \leq 4 \cdot \max_{S \subseteq V} \left(\frac{e(S)}{e(G)}-\frac{\vol(S)^2}{\vol(G)^2}\right).
	\]
\end{corollary}
\begin{proof}
	Consider $2$-element partitions of $V$. We have 
	\[
	\begin{split}
		\max_{\mathcal{A}: |\mathcal{A}| = 2} \modul_{\mathcal{A}}(G) & = \max_{\mathcal{A} = \{S, V \setminus S\}} \left(\frac{e(S)}{e(G)}-\frac{\vol(S)^2}{\vol(G)^2}+\frac{e(V \setminus S)}{e(G)}-\frac{\vol(V \setminus S)^2}{\vol(G)^2}\right) \\
		& \leq 2 \cdot \max_{S \subseteq V} \left(\frac{e(S)}{e(G)}-\frac{\vol(S)^2}{\vol(G)^2}\right).
	\end{split}
	\]
	The conclusion follows by Lemma \ref{lemma:only2parts}. (Note that for $S=V$ the argument of the maximum equals $0$ thus the bound is non-negative.)
\end{proof}

The above corollary frees us from considering all the partitions of $V$ when analyzing modularity of $G_n^h$. We can simply concentrate on upper bounding the values of $\left(\frac{e(S)}{e(G_n^h)}-\frac{\vol(S)^2}{\vol(G_n^h)^2}\right)$ over all $S \subseteq V$. To do so, we use the concentration results for $e(S)$ and $\vol(S)$ obtained in Section \ref{sec:concentrations}.

\begin{proof} [Proof of Theorem \ref{thm:main}]
	Fix $\varepsilon>0$ and let $g_{\mathcal{E}}(h) = \frac{g_{\mathcal{V}}(h)}{2} + \sqrt{2 \ln{2}}$. Let us define the events $\mathcal{H}$, $\mathcal{E}$ and $\mathcal{V}$ as follows
	\[
	\mathcal{H} = \left\{ \modul(G_n^h) \leq \frac{(1+\varepsilon)f(h)}{\sqrt{h}}\right\},
	\]
	\[
	\mathcal{E} = \left\{\forall S \subseteq V \, e(S) \leq \mu(\tilde{S})^2 +  (1+\varepsilon)g_{\mathcal{E}}(h) \frac{M}{\sqrt{h}}	\right\},
	\]
	\[
	\mathcal{V} = \left\{\forall S \subseteq V \, \vol(S) \geq 2\sqrt{M}\mu(\tilde{S}) - (1+\varepsilon)g_{\mathcal{V}}(h)\frac{M}{\sqrt{h}}\right\}.
	\]
	By Theorem \ref{thm:vol_concentration} and Theorem~\ref{thm:edges_concentration} we know that $\Pr[\mathcal{E} \cap \mathcal{V}] = 1-o(1)$. Thus it is enough to show that the event $\mathcal{E} \cap \mathcal{V}$ implies the event $\mathcal{H}$.
	
	Assume that $\mathcal{E} \cap \mathcal{V}$ holds. Recall that $e(G_n^h) = M$ and $\vol(G_n^h)=2M$. For any $S \subseteq V$ we may write
	\[
	\begin{split}
		\frac{e(S)}{e(G_n^h)} &-\frac{\vol(S)^2}{\vol(G_n^h)^2}  = \frac{4Me(S) - \vol(S)^2}{4M^2} \\
		& \leq \frac{1}{4M^2}\left(4M\mu(\tilde{S})^2 +  4 (1+\varepsilon)g_{\mathcal{E}}(h) \frac{M^2}{\sqrt{h}} 
		- \left( 2\sqrt{M}\mu(\tilde{S}) - (1+\varepsilon)g_{\mathcal{V}}(h) \frac{M}{\sqrt{h}} \right)^2 \right) \\
		& = \frac{ (1+\varepsilon)g_{\mathcal{E}}(h)}{ \sqrt{h}} + \frac{\mu(\tilde{S})(1+\varepsilon)g_{\mathcal{V}}(h)}{\sqrt{Mh}} - \frac{(1+\varepsilon)^2 g_{\mathcal{V}}(h)^2}{4 h}\\
		& \leq \frac{ (1+\varepsilon)\left(g_{\mathcal{E}}(h) + g_{\mathcal{V}}(h) - g_{\mathcal{V}}(h)^2/(4\sqrt{h})\right)}{ \sqrt{h}},
	\end{split}
	\]
	where the last inequality follows from Lemma \ref{lemma:mu} and is valid for sufficiently large~$n$.
	By~Corollary \ref{cor:mod_upperbound}, for sufficiently large $n$ we obtain
	\[
	\begin{split}
		\modul{(G_n^h)} & \leq 4 \cdot \max_{S \subseteq V} \left(\frac{e(S)}{e(G_n^h)}-\frac{\vol(S)^2}{\vol(G_n^h)^2}\right) \\
		& \leq \frac{ (1+\varepsilon)\left(4g_{\mathcal{E}}(h) + 4g_{\mathcal{V}}(h) - g_{\mathcal{V}}(h)^2/\sqrt{h}\right)}{ \sqrt{h}} = \frac{ (1+\varepsilon)f(h)}{ \sqrt{h}}.
	\end{split}
	\] 
	We got that for sufficiently large $n$ the event $\mathcal{E} \cap \mathcal{V}$ implies the event~$\mathcal{H}$ thus $\Pr[\mathcal{E} \cap \mathcal{V}] = 1 - o(1)$ implies $\Pr[\mathcal{H}] = 1-o(1)$ and the proof is finished.
\end{proof}

We finish by proving  Corollary~\ref{cor:main}.

\begin{proof}[Proof of Corollary \ref{cor:main}]
	The corollary follows from Theorem \ref{thm:main} by considering  $\tilde{f}(h) = 6 g_{\mathcal{V}}(h) + 4 \sqrt{2 \ln{2}}$ instead of $f(h)$ in the upper bound (note that $3 \sqrt{2\ln{2}} \leq 3.54$, $(8/9)\ln{2} \leq 0.62$ and $ 4 \ln{2} + 12 + 4 \sqrt{2\ln{2}} \leq 19.49$).
\end{proof}


\section{Concluding remarks} \label{sec:conclusion}

We showed that the modularity of a preferential attachment graph $G_n^h$ is, with high probability, upper bounded by a function of the order $\Theta(\sqrt{\ln{h}}/\sqrt{h})$. This proves Conjecture \ref{conj:mod->0} but means that Conjecture \ref{conj:mod_theta}, saying that modularity of $G_n^h$ is, with high probability, of the order $\Theta(1/\sqrt{h})$, still remains open. Note that the term $\Theta(1/\sqrt{h})$ can also be seen as $\Theta(1/\sqrt{\bar{d}_h})$, where $\bar{d}_h$ states for the average vertex degree in $G_n^h$. Such behavior of modularity has already been reported in other random graphs. For $G_{n,r}$ being a random $r$-regular graph it is known that whp $\modul(G_{n,r}) = \Theta(1/\sqrt{r})$ (see \cite{DiSk18}). 
For binomial random graph $G(n,p)$ it is known that $\modul(G(n,p)) = \Theta(1/\sqrt{np})$  when $np \geq 1$ and $p$ is bounded below $1$ (see \cite{DiSk20,RybarczykSulkowska_Gnp}). These might be premises supporting Conjecture \ref{conj:mod_theta}.

To the best of our knowledge, this paper provides the first concentration results for $\vol(S)$, $e(S)$, and $e(S,V\setminus S)$ in $G_n^h$, where $S$ can be an arbitrary subset of $V$. The analogous results obtained so far, e.g. in \cite{Bol_Rio_diameter}, \cite{FrPePrRe20} or \cite{prokhorenkova2017modularity_internet_Math} always addressed ``compact'' subsets of vertices, i.e., sets of the form $\{i, i+1, i+2, \ldots, j\}$. (In Lemma~4 of \cite{FrPePrRe20} the authors investigate the volume of any set $S \subseteq [t]$ of size $1 \leq k \leq t$ at time $t$ but in fact in their proof the volume of $S$ is upper bounded by the volume of the $k$ ``oldest'' vertices in $[t]$, i.e., the volume of the set $[k]$.) In this paper more accurate analysis was possible thanks to introducing indicator functions $\delta_i^{\tilde{S}}$ in Definition \ref{def:mu_delta_c}.  We believe that the proof methods leading to results obtained in Section~\ref{sec:concentrations} might help in the future to get bounds for edge expansion or conductance of $G_n^h$ that are stronger than those currently known.

\vspace{20pt}

\subsection*{Acknowledgments}This research was funded in whole or in part by National Science Centre, Poland, grant OPUS-25 no 2023/49/B/ST6/02517. For the purpose of Open Access, the authors have applied a CC-BY public copyright licence to any Author Accepted Manuscript (AAM) version arising from this submission.



\bibliographystyle{plain}
\bibliography{modularity_PA}

\newpage
\appendix

\section{Auxiliary lemmas - proofs}\label{sec:itemStyles} \label{app:A}

%
%
%

\begin{proof}[Proof of Lemma~\ref{lemma:only_delta}]
	First note that by Lemma~\ref{lemma:harmonic} 
	\[
	\sum_{i=1}^{t_0} \frac{{\delta}_i}{2i-1}\le \sum_{i=1}^{t_0} \frac{1}{2i-1}=O(\ln t_0). 
	\]
	Therefore
	\begin{equation} \label{eq:lnt0}
		\sum_{i=t_0+1}^{M} \frac{{\delta}_i}{2i-1}=\sum_{i=1}^{M} \frac{{\delta}_i}{2i-1}+O(\ln t_0).
	\end{equation}
	For $i \in \{2, \ldots, M\}$, by Lemma \ref{lemma:ct} we have
	\[
	\frac{1}{2i-1} \leq \frac{1}{2(i-1)} \leq \frac{\pi}{2} c_{i-1}^2 (1+O(1/i))
	\]
	and note that the constant hidden in $O(1/i)$ is the same for all $i$'s. Therefore, by the fact that ${\delta}_i \in\{0,1\}$ and  $c_i \leq 1$,
	\[ 
	\begin{split}
		\sum_{i=1}^{M} \frac{{\delta}_i}{2i-1} & \leq \delta_1 + \sum_{i=2}^{M} \frac{\pi}{2} {\delta}_i c_{i-1}^2 + \frac{\pi}{2} \sum_{i=2}^{M} O(1/i) \leq \frac{\pi}{2} \sum_{i=1}^{M} ({\delta}_i c_{i-1})^2  + O(\ln{M}),
	\end{split}
	\]
	which, together with (\ref{eq:lnt0}) and the fact that $t_0 \in [M]$, yields the desired upper bound.
	
	On the other hand, for $i \in \{2, \ldots, M\}$, by Lemma \ref{lemma:ct} we have
	\[
	\frac{1}{2(i-1)} \geq \frac{\pi}{2} c_{i-1}^2
	\]
	thus
	\[
	\begin{split}
		\sum_{i=1}^{M} \frac{{\delta}_i}{2i-1} 
		& \geq \sum_{i=2}^{M} \frac{2i-2}{2i-1} \frac{{\delta}_i}{2i-2} = \sum_{i=2}^{M} \frac{{\delta}_i}{2(i-1)}-\sum_{i=2}^{M}\frac{{\delta}_i}{(2i-1)(2i-2)}
		\\
		&\geq \sum_{i=1}^{M} \frac{\pi}{2} {\delta}_i c_{i-1}^2 - O(\ln{M})
		= \frac{\pi}{2} \sum_{i=1}^{M} ({\delta}_i c_{i-1})^2  - O(\ln{M}),
	\end{split}
	\]
	which, together with (\ref{eq:lnt0}) and the fact that $t_0 \in [M]$, yields the desired lower bound.
\end{proof}

\begin{proof}[Proof of Lemma~\ref{lemma:Zi_first_part}]
	First note that by Lemma~\ref{lemma:mu}
	\[
	\sum_{i=1}^{t_0} {\delta}_i \frac{2 \sqrt{i-1} \mu(\tilde{S}_{i-1})}{2i-1}\le \sum_{i=1}^{t_0} \frac{2 \sqrt{i-1} (\sqrt{i-1}+\frac12)}{2i-1}\le 2t_0.
	\]
	Therefore
	\begin{equation} \label{eq:t0}
		\sum_{i=t_0+1}^{M} {\delta}_i \frac{2 \sqrt{i-1} \mu(\tilde{S}_{i-1})}{2i-1}
		=\sum_{i=1}^{M} {\delta}_i \frac{2 \sqrt{i-1} \mu(\tilde{S}_{i-1})}{2i-1}+O(t_0).
	\end{equation}
	Moreover, for $i \in \{2, \ldots, M\}$, by Lemma \ref{lemma:ct} we have 
	\[
	\frac{1}{\sqrt{i-1}} \leq \sqrt{\pi} c_{i-1} (1+O(1/i)) \quad \quad \textnormal{and} \quad \quad 	c_{i-1} \leq \frac{1}{\sqrt{\pi} \sqrt{i-1}}
	\]
	and note that the constant hidden in $O(1/i)$ is the same for all $i$'s. Therefore, by the definition of $\mu(\tilde{S}_i)$, Lemma \ref{lemma:mu}, and the fact that ${\delta}_i \in \{0,1\}$
	\[
	\begin{split}
		\sum_{i=1}^{M} {\delta}_i & \frac{2 \sqrt{i-1} \mu(\tilde{S}_{i-1})}{2i-1} = \sum_{i=2}^{M} {\delta}_i \frac{2 \sqrt{i-1} \mu(\tilde{S}_{i-1})}{2i-1} \leq \sum_{i=2}^{M} {\delta}_i \frac{\mu(\tilde{S}_{i-1})}{\sqrt{i-1}}\\
		& \leq \sum_{i=2}^{M} {\delta}_i \sqrt{\pi} c_{i-1} \mu(\tilde{S}_{i-1}) (1+O(1/i))\\
		& \leq \sum_{i=2}^{M} \left( {\delta}_i \sqrt{\pi} c_{i-1} \frac{\sqrt{\pi}}{2}\sum_{j=1}^{i-1} {\delta}_j c_{j-1} \right) +  \sum_{i=2}^{M} \sqrt{\pi} \frac{\sqrt{i-1} + \frac{1}{2}}{\sqrt{\pi}\sqrt{i-1}} O(1/i) \\
		& = \frac{\pi}{2} \sum_{i=1}^{M} \left( {\delta}_i c_{i-1} \sum_{j=1}^{i-1} {\delta}_j c_{j-1} \right) + O(\ln{M}),
	\end{split}
	\]
	which, together with (\ref{eq:t0}), yields the desired upper bound.
	
	On the other hand, again by Lemma \ref{lemma:ct}, the definition of $\mu(\tilde{S}_i)$, Lemma \ref{lemma:mu}, and the fact that ${\delta}_i \in \{0,1\}$
	\[
	\begin{split}
		\sum_{i=1}^{M} {\delta}_i & \frac{2 \sqrt{i-1} \mu(\tilde{S}_{i-1})}{2i-1} = \sum_{i=2}^{M} \left(1-\frac{1}{2i-1}\right) {\delta}_i \frac{2 \sqrt{i-1} \mu(\tilde{S}_{i-1})}{2i-2}  \\
		& = \sum_{i=2}^{M} {\delta}_i \frac{\mu(\tilde{S}_{i-1})}{\sqrt{i-1}} - \sum_{i=2}^{M}  \frac{\delta_i}{2i-1}\frac{ \mu(\tilde{S}_{i-1})}{\sqrt{i-1}}\\
		& \geq \sum_{i=2}^{M} \left( {\delta}_i \sqrt{\pi} c_{i-1} \frac{\sqrt{\pi}}{2}\sum_{j=1}^{i-1} {\delta}_j c_{j-1} \right) - \sum_{i=2}^{M}  \frac{1}{2i-1}\frac{\sqrt{i-1}+\frac{1}{2}}{\sqrt{i-1}} \\
		& = \frac{\pi}{2} \sum_{i=1}^{M} \left( {\delta}_i c_{i-1} \sum_{j=1}^{i-1} {\delta}_j c_{j-1} \right) - O(\ln{M}),
	\end{split}
	\]
	which, together with (\ref{eq:t0}), yields the desired lower bound.
\end{proof}

\begin{proof}[Proof of Lemma~\ref{lemma:Zi_second_part}]
	By the fact that ${\delta}_i \in \{0,1\}$ we immediately get
	\[
	\sum_{i=t_0+1}^{M} {\delta}_i \frac{C(i-1)}{2i-1} \leq \frac{C}{2} \sum_{i=t_0+1}^{M} {\delta}_i \leq \frac{C}{2} (M-t_0).
	\]
\end{proof}

\newpage

\section{Edge density and volume results for {\boldmath{$G_n^h$}} - supplementary proofs} \label{app:B}

\begin{proof}[Supplement to the proof of Theorem~\ref{thm:vol_concentration_all}]
	To get the opposite bound we repeat the reasoning for the martingale $-\hat{Z}_{t_0}, \ldots, -\hat{Z}_t$. We apply again Freedman's inequality (Lemma \ref{lemma:Freedman}) with $A = \frac{2}{\sqrt{\pi (t/h)}}(1+O(1/t))$, $W = \frac{1}{4 \pi} \ln{h} + O(1/t)$ and $\lambda =  \frac{(1+\varepsilon)\bar{g}_{\mathcal{V}}(h)}{\sqrt{\pi}} \sqrt{t/h}$, where (as before) $\bar{g}_{\mathcal{V}}(h) = g_{\mathcal{V}}(h)-2$. For sufficiently large $t$ and $n$ we get
	\[
	\begin{split}
		\Pr\left[-\hat{Z}_t \vphantom{\frac{{C}_{3\varepsilon}}{\sqrt{\pi}}}\right.& \left. \geq -\hat{Z}_{t_0} + \frac{(1+\varepsilon)\bar{g}_{\mathcal{V}}(h)}{\sqrt{\pi}} \sqrt{t/h} \right] \leq \exp\left\{-\frac{(1+\varepsilon/2)^2\bar{g}_{\mathcal{V}}(h)^2 }{\frac{1}{2} \ln{h} + \frac{4}{3} (1+\varepsilon/2)\bar{g}_{\mathcal{V}}(h)} \cdot \frac{t}{h}\right\},
	\end{split}
	\]
	which for sufficiently large $t$ and $n$ implies
	\begin{equation} \label{eq:prob_hatZ-}
		\Pr\left[\hat{Z}_t \leq \hat{Z}_{t_0} - \frac{(1+\varepsilon)\bar{g}_{\mathcal{V}}(h)}{\sqrt{\pi}}\sqrt{t/h}\right] \leq 2^{-(1+{\varepsilon}/2)t/h}.
	\end{equation}
	Now, the complementary event $\left\{\hat{Z}_t \geq \hat{Z}_{t_0} - \frac{(1+\varepsilon)\bar{g}_{\mathcal{V}}(h)}{\sqrt{\pi}} \sqrt{t/h}\right\}$ is equivalent to
	\[
	\left\{Z_t - \frac{1}{c_t}\sum_{i=1}^{t}{\delta}_i c_{i-1} \geq \frac{c_{t_0}}{c_t} Z_{t_0} - \frac{1}{c_t}\sum_{i=1}^{t_0}{\delta}_i c_{i-1} - \frac{1}{c_t} \frac{(1+\varepsilon)\bar{g}_{\mathcal{V}}(h)}{\sqrt{\pi}} \sqrt{t/h}\right\}.
	\]
	This time by the definition of $\mu(\tilde{S}_t)$ and Lemma \ref{lemma:ct} we have
	\begin{equation} \label{eq:bound-_mu}
		\frac{1}{c_t}\sum_{i=1}^{t}{\delta}_i c_{i-1} \geq 2 \sqrt{t} \mu(\tilde{S}_t),
	\end{equation}
	and by the definition of $\mu(\tilde{S}_t)$, Lemma \ref{lemma:ct} and Lemma \ref{lemma:mu}
	\begin{equation} \label{eq:bound-_mut0}
		\frac{1}{c_t}\sum_{i=1}^{t_0}{\delta}_i c_{i-1} \leq 2 \sqrt{t} \mu(\tilde{S}_{t_0}) e^{\frac{1}{8t}+\frac{1}{4 \cdot 144 t^2}} \leq \frac{2t}{\sqrt{h}} + O(1).
	\end{equation}
	Hence, by (\ref{eq:bound-_mu}), (\ref{eq:bound-_mut0}) and (\ref{eq:bound+_x_Azuma}), the event  $\left\{\hat{Z}_t \geq \hat{Z}_{t_0} - \frac{(1+\varepsilon)\bar{g}_{\mathcal{V}}(h)}{\sqrt{\pi}} \sqrt{t/h}\right\}$  implies
	\[
	\left\{Z_t - 2 \sqrt{t}\mu(\tilde{S}_t) \geq - \frac{2t}{\sqrt{h}} - (1+\varepsilon)\bar{g}_{\mathcal{V}}(h) \frac{t}{\sqrt{h}} - O(1)\right\}.
	\]
	Therefore by (\ref{eq:prob_hatZ-}) for sufficiently large $t$ and $n$ we get
	\begin{equation} \label{eq:P_bound-}
		\Pr\left[\vol_{T_t}(\tilde{S}_t) - 2 \sqrt{t}\mu(\tilde{S}_t) \leq - (1+\varepsilon)g_{\mathcal{V}}(h) \frac{t}{\sqrt{h}} \right] \leq 2^{-(1+{\varepsilon}/2)t/h}.
	\end{equation}
	
	Finally, by (\ref{eq:P_bound+}) and (\ref{eq:P_bound-}), for sufficiently large $t$ and $n$ we may write
	\[
	\Pr\left[\left|\vol_{T_t}(\tilde{S}_t) - 2 \sqrt{t} \mu(\tilde{S}_t)\right| \geq (1+\varepsilon)g_{\mathcal{V}}(h) \frac{t}{\sqrt{h}} \right] \leq 2\cdot 2^{-(1+{\varepsilon/2})t/h}.
	\]
\end{proof}


\begin{proof}[Supplement to the proof of Lemma~\ref{lemma:XM_as_Zis}]
	To get the opposite bound we repeat the reasoning for the martingale $-\hat{X}_{t_0}, \ldots, -\hat{X}_M$. We have
	\[
	\Pr\left[\hat{X}_M \leq \hat{X}_{t_0} - {B}_{\varepsilon} \frac{M}{\sqrt{h}}\right] \leq 2^{-(1+\varepsilon)n},
	\]
	which implies
	\begin{equation} \label{eq:XM_upper}
		\Pr\left[X_M - \left( X_{t_0} + \sum_{i=t_0+1}^{M} {\delta}_i\frac{Z_{i-1}+1}{2i-1} \right)  \leq - {B}_{\varepsilon} \frac{M}{\sqrt{h}}\right] \leq 2^{-(1+\varepsilon)n}.
	\end{equation}
	By (\ref{eq:XM_lower}) and (\ref{eq:XM_upper}), for a fixed $S \subseteq V$ we get
	\begin{equation} \label{eq:XM_final}
		\Pr\left[\left|X_M - \left(X_{t_0} + \sum_{i=t_0+1}^{M} {\delta}_i \frac{Z_{i-1}+1}{2i-1} \right)\right| \geq {B}_{\varepsilon} \frac{M}{\sqrt{h}}
		\right] \leq 2 \cdot 2^{-(1+\varepsilon) n}.
	\end{equation}
	Now, for $S \subseteq V$, let the event $\mathcal{E}_S$ be defined as
	\[
	\mathcal{E}_S = \left\{\left|e(S) - \left(e(S_{t_0/h}) + \sum_{i=t_0+1}^{M} {\delta}_i \frac{Z_{i-1}+1}{2i-1} \right) \right| \leq {B}_{\varepsilon} \frac{M}{\sqrt{h}} \right\}.
	\]
	By (\ref{eq:XM_final}) and the union bound we get (recall that $X_M = e(S)$ and $X_{t_0} = e(S_{t_0/h})$)
	\[
	\begin{split}
		\Pr\left[\forall S \subseteq V \, \mathcal{E}_S\right] & = 1 - \Pr\left[\exists S \subseteq V \, \mathcal{E}_S^C\right] \geq  1 - \sum_{S \subseteq V} \Pr\left[\mathcal{E}_S^C\right] \\
		& \geq 1 - 2^n \cdot 2 \cdot 2^{-(1+ \varepsilon)n} = 1 -  2 \cdot 2^{-\varepsilon n} = 1-o(1).
	\end{split}
	\]
\end{proof}

\end{document}